\documentclass[]{interact}

\usepackage{algorithm}
\usepackage{algorithmic}
\usepackage{float}
\usepackage{xcolor}
\usepackage{hyperref}
\usepackage{graphicx}

\usepackage{epstopdf}
\usepackage[caption=false]{subfig}

\usepackage[numbers,sort&compress]{natbib}
\bibpunct[, ]{[}{]}{,}{n}{,}{,}
\renewcommand\bibfont{\fontsize{10}{12}\selectfont}
\makeatletter
\def\NAT@def@citea{\def\@citea{\NAT@separator}}
\makeatother

\theoremstyle{plain}
\newtheorem{theorem}{Theorem}[section]
\newtheorem{lemma}[theorem]{Lemma}
\newtheorem{corollary}[theorem]{Corollary}

\theoremstyle{definition}
\newtheorem{definition}[theorem]{Definition}
\newtheorem{example}[theorem]{Example}

\theoremstyle{remark}
\newtheorem{remark}{Remark}

\begin{document}


\title{Accelerated Bregman gradient methods for relatively smooth and relatively Lipschitz continuous minimization problems}

\author{
\name{O.~S.~Savchuk\textsuperscript{a,b}\thanks{O.~S.~Savchuk. Email: oleg.savchuk19@mail.ru}, M.~S.~Alkousa\textsuperscript{c,d}  \thanks{M.~S.~Alkousa. Email: m.alkousa@innopolis.ru}, A.~S.~Shushko\textsuperscript{d}\thanks{A.~S.~Shushko. Email: shushko.and@gmail.com},    A.~A.~Vyguzov\textsuperscript{d,b}\thanks{A.~A.~Vyguzov. Email: avyguzov@yandex.ru}, F.~S.~Stonyakin\textsuperscript{d,a,c}\thanks{F.~S.~Stonyakin. Email: fedyor@mail.ru}, D.~A.~Pasechnyuk\textsuperscript{e,f} \thanks{D.~A.~Pasechnyuk. Email: pasechnyuk2004@gmail.com} and A.~V.~Gasnikov\textsuperscript{c,d,b} \thanks{A.~V.~Gasnikov. Email: gasnikov@yandex.ru}}
\affil{\textsuperscript{a}V. I. Vernadsky Crimean Federal University, Vernadsky Avenue, 4, Simferopol, Russia; \textsuperscript{b}Adyghe State University, 208 Pervomayskaya st., Maykop, Russia; \textsuperscript{c}Innopolis University, Innopolis, Universitetskaya Str., 1, 420500, Russia; \textsuperscript{d}Moscow Institute of Physics and Technology, 9 Institutsky lane, Dolgoprudny, 141701, Russia; \textsuperscript{e}Mohamed bin Zayed University of Artificial Intelligence, Abu Dhabi, UAE;
\textsuperscript{f}Ivannikov Institute for System Programming of the Russian Academy of Sciences, Russia
}}

\maketitle

\begin{abstract}
In this paper, we propose some accelerated methods for solving optimization problems under the condition of relatively smooth and relatively Lipschitz continuous functions with an inexact oracle.  We consider the problem of minimizing the convex differentiable and relatively smooth function concerning a reference convex function. The first proposed method is based on a similar triangles method with an inexact oracle, which uses a special triangular scaling property for the used Bregman divergence. The other proposed methods are non-adaptive and adaptive (tuning to the relative smoothness parameter) accelerated Bregman proximal gradient methods with an inexact oracle. These methods are universal in the sense that they are applicable not only to relatively smooth but also to relatively Lipschitz continuous optimization problems.  We also introduced an adaptive intermediate Bregman method which interpolates between slower but more robust algorithms non-accelerated and faster, but less robust accelerated algorithms. We conclude the paper with the results of numerical experiments demonstrating the advantages of the proposed algorithms for the Poisson inverse problem.
\end{abstract}


\section{Introduction}
With the increase in the number of applications that can be modeled as large-scale (or even huge-scale) optimization problems (some of such applications arising in machine learning, deep learning, data science, control, signal processing, statistics, and so on), first-order methods, which require low iteration cost as well as low memory storage, have received much interest over the past few decades to solve the optimization problems, in the smooth and non-smooth cases \cite{Beck2017First,nest_lec}.  Recently, there was introduced a new direction for the research, associated with the development of gradient-type methods for optimization problems with relatively Lipschitz-continuous \cite{lu2018relativecontinuity}, relatively smooth \cite{Bauschke2017first} and relatively strongly convex \cite{Lu2018Relatively} functions. Such methods are urgently in high demand due to numerous theoretical and applied problems. For example, the D-optimal design problem (see Example \ref{ex:DOptDes}), and the Poisson inverse problem (see Example \ref{ex:InvPois}) turned out to be relatively smooth \cite{Lu2018Relatively,hanzely2021accelerated}. It is also quite interesting that in recent years there have been appeared many applications of these methods to solve auxiliary problems that arise in tensor methods for convex minimization problems of the second and higher orders \cite{Nest_tens,Nest_core}. 

Many real-world optimization problems involve minimizing relatively smooth functions using an inexact oracle that provides noisy or approximate access to the functions and gradient information. Some examples of such problems include training machine learning models on large datasets, where evaluating the full information concerning the objective function (its value and gradient at iterated points) is computationally expensive, so stochastic gradients are used instead, optimizing hyperparameters of machine learning models using validation set performance, which is a noisy estimate of the true objective, tuning the parameters of physical systems or simulations that have inherent noise or uncertainty \cite{bottou2018optimization,bergstra2012random}. In this context, the concept of relative smoothness plays a crucial role. This generalization of the standard smoothness condition allows for more flexibility in the choice of the reference function, which can lead to improved computational guarantees for first-order optimization methods. 

The concept of relative smoothness first introduced in \cite{Bauschke2017first}, it is an important condition that generalizes the standard Lipschitz continuity of the gradient, allows for the analysis of convergence of gradient methods for a broader class of functions, and enables efficient optimization in many applications. Further lower bounds for relatively smooth problems were established in \cite{rel_smooth_review}. It can be effectively used in machine learning problems, where the concept is applied to analyze the convergence of gradient methods when training logistic regression models with non-Euclidean regularizers. In signal processing, relative smoothness is used to study the convergence of proximal splitting methods for solving composite optimization problems. In the seminal work \cite{Lu2018Relatively}, the authors demonstrated the versatility of the relative smoothness condition by applying it to various settings of optimization problems, including the D-Optimal Design Problem and the minimization of the Volumetric Barrier Function. They showed that by exploiting the relative smoothness structure of these problems, it is possible to design efficient gradient methods with improved convergence rates.

Motivated by applications in machine learning and engineering, where the presence of noise or uncertainty is common, the concept of an inexact oracle, where the function and gradient values have some error, can significantly degrade the performance of accelerated gradient methods. In some cases, non-accelerated methods that are more robust to noise can outperform accelerated ones. Recently the authors of \cite{Devolder2014First} proposed an ''intermediate'' method that interpolates between accelerated and non-accelerated convergence rates, using a special parameter $p$ that controls the degree of acceleration, depending on the current accumulation of the solution estimation inaccuracy. When $p$ tends to $2$, it recovers the accelerated method, while $p$ tends to $1$, gives the non-accelerated convergence rate, but at the same time it guarantees the achievement of a higher accuracy of the solution to the problem (see the second term of Remark \ref{rem:exact_setting}).

In this paper, we focus on developing similar previously mentioned optimization methods for relatively smooth problems with an inexact oracle, which can be considered the main novelty of the work. For the convex optimization problem, we propose two new families of methods, the first is an adaptive fast gradient method with an inexact oracle. It is based on the similar triangles method \cite{stonyakin2021inexact}, and exploits the idea of an inexact oracle and the triangular scaling property of the Bregman divergence corresponding to the used prox structure of the problem. This proposed algorithm is universal in the sense that it is applicable not only to relatively smooth but also to relatively Lipschitz continuous optimization problems. For the second family, we propose a different picture, where we propose non-adaptive and adaptive accelerated Bregman proximal gradient methods with an inexact oracle. These methods are also universal in the sense that they are applicable not only to relatively smooth but also to relatively Lipschitz continuous optimization problems. Further, based on \cite{devolder2013intermediate,dvurechensky2016stochastic,kamzolov2021universal}, we propose an intermediate method parameterized by $p\in [1,2]$ and interpolating between non-accelerated and accelerated methods. The proposed method is provided by an adaptive way of choosing parameters $p$ and $L$ (relative smoothness parameter) along the iteration process, which leads to better robustness by slowing down the convergence rate.

\subsection{Contributions}
To sum it up, the contributions of the paper can be formulated as follows.
\begin{enumerate}
    \item We proposed an adaptive fast gradient method with an inexact oracle. We concluded its convergence rate for the convex optimization problems and proved that it is applicable not only to relatively smooth but also to relatively Lipschitz continuous optimization problems.

    \item We proposed non-adaptive and adaptive accelerated Bregman proximal gradient methods with an inexact oracle.  We concluded their convergence rate and also proved that they are applicable not only to relatively smooth but also to relatively Lipschitz continuous optimization problems. 

    \item We proposed an automatically adaptive intermediate Bregman method, that interpolates between non-accelerated and accelerated methods, with an inexact oracle, and concluded its convergence rate.  

    \item  We provided some numerical experiments for the Poisson inverse problem, with a comparison of the proposed methods with other known methods. 
\end{enumerate}

\subsection{Paper organization}
The paper consists of an introduction and 6 main sections. In Sect. \ref{sect:basics} we mentioned the statement of the problem, some fundamental concepts (such as Bregman divergence, relative smoothness, inexact oracle, triangular scaling property), and some examples. Sect. \ref{sect:FGM} is devoted to an adaptive fast gradient method using an inexact oracle with analysis of its convergence. In Sect. \ref{sect:auccBPGM} we proposed non-adaptive and adaptive accelerated Bregman proximal gradient methods with an inexact oracle and we concluded their convergence rate.  Sect. \ref{sec:AIBM} is devoted to an adaptive intermediate method, with analysis of its convergence, that interpolates between non-accelerated and accelerated methods. In Sect. \ref{sect:numerical_exper} we presented the results of some numerical experiments, these results demonstrate the efficiency of the proposed methods for the Poisson inverse problem. In the last section, we briefly presented the results that were reached in the paper.

\section{Problem statement, basic concepts and definitions}\label{sect:basics}

Let $\mathbb{E}$ be an $n$-dimensional real vector space and $\mathbb{E}^*$ be its dual. We denote the value of a linear function $g \in \mathbb{E}^*$ at $x\in \mathbb{E}$ by $\langle g, x \rangle$. Let $\|\cdot\|$ be some norm on $\mathbb{E}$, $\|\cdot\|_{*}$ be its dual, defined by $\|g\|_{*} = \max\limits_{x} \big\{ \langle g, x \rangle, \| x \| \leq 1 \big\}$. $\|\cdot\|_2$ denotes the standard Euclidean norm.  We use $\nabla f(x)$ to denote the gradient of a differentiable function $f$ at a point $x \in {\rm dom} f$ and, with a slight abuse of notation, a subgradient of a convex function $f$ at a point $x \in {\rm dom} f$.

Let $Q \subset \mathbb{E}$ be a closed convex set, $f: Q \longrightarrow \mathbb{R}$ be a convex differentiable on an open set that contained the relative interior of $Q$ (denoted as ${\rm rint} Q$). 

In this paper, we consider the following optimization problem
\begin{equation}\label{main_composite_problem}
    \min\limits_{x\in Q} f(x). 
\end{equation}

For solving the problem \eqref{main_composite_problem}, we can consider the \textit{proximal gradient method}, which, by taking $x_0 \in {\rm rint} Q$ as an initial point, has the following form 
\begin{equation}\label{eq_PGM}
    x_{k+1} = \arg\min\limits_{x\in Q} \left\{f(x_k) + \left\langle \nabla f(x_k), x - x_k \right\rangle + \frac{L}{2} \|x - x_k \|^2\right\}, \; \; \forall k \geq 0,
\end{equation}
with $L > 0$ (see \eqref{eq_3}). The assumption that $Q$ is simple means that the minimization problem in \eqref{eq_PGM} can be solved efficiently, especially if it admits a closed-form solution. 

In \eqref{eq_PGM}, we use the gradient of the function $f$ at $x_k$, i.e., $\nabla f(x_k)$, to construct a local quadratic approximation of $f$ near $x_k$. Assuming that the function $f$ is bounded below, the convergence of the \textit{proximal gradient method} can be established
if $f(x_{k+1}) \leq  f(x_k)$ for all $k \geq 0 $. A sufficient condition for this to hold is that the quadratic approximation of $f$ in \eqref{eq_PGM} is an upper approximation. This is the main idea of many general methods for nonlinear optimization, where a common assumption is for the gradient of $f$ to satisfy a uniform Lipschitz condition. That is, there exists $L > 0 $ such that the following inequality holds
\begin{equation}\label{eq_3}
    \| \nabla f(x) - \nabla f(y) \|_* \leq L \| x - y \|,  \quad  \forall x, y \in { \rm rint} Q.
\end{equation}

From \eqref{eq_3}, it follows that \cite{nest_lec} 
\begin{equation}\label{eq_4}
    f(y) \leq f(x) + \langle \nabla f(x), y - x \rangle + \frac{L}{2} \|y - x \|^2,  \quad  \forall x \in {\rm rint} , y \in Q.
\end{equation}

For the \textit{proximal gradient method} \eqref{eq_PGM}, it can be shown  that its convergence rate will be \cite{komposite,komposite_2}
\begin{equation}\label{eq_40_}
    f(x_k) - f(x) \leq \frac{L \|x - x_0 \|^2}{ 2 k} = \mathcal{O}\left(\frac{1}{k}\right), \quad  \forall x \in Q, \; \text{and} \; \forall k \geq 1.
\end{equation}

Under the same assumptions, \textit{accelerated proximal gradient methods} can achieve a faster convergence rate \cite{komposite,komposite_acc} , which is
\begin{equation} \label{eq_400__}
    f(x_k) - f(x) \leq \frac{2 L \|x - x_0 \|^2}{(k + 2)^2}  = \mathcal{O}\left(\frac{1}{k^2}\right), \quad  \forall x \in Q, \; \text{and} \; \forall k \geq 0 . 
\end{equation}
This convergence rate is optimal (up to a constant factor) for a given class of convex optimization problems \cite{compos_optim,nest_lec}.

Although the uniform smoothness condition \eqref{eq_3} is central to the design and analysis of first-order methods, there are many applications where the objective function does not have this property, despite being convex and differentiable. For example, in the Poisson inverse problem \cite{Entropy,poi} (see Example \ref{ex:InvPois}) and the D-optimal design problem \cite{D-optim,D-opt} (see Example \ref{ex:DOptDes}), the objective functions involve a logarithm in the form of the log-determinant or relative entropy, whose gradients may increase sharply toward the boundary of the feasible region. To design efficient first-order algorithms for such problems, the notion of \textit{relative smoothness} has been introduced in several works (see e.g. \cite{bregman,Bauschke2017first,Lu2018Relatively}).

\subsection{Bregman divergence and relative smoothness}

Let $d : Q \longrightarrow \mathbb{R}$  be a distance-generating function (also called prox-function or reference function) that is continuously differentiable and $1$-strongly convex.

\begin{definition}\label{def_con}
The Bregman divergence (also called Bregman distance) associated with the prox-function $d$ is defined as 
\begin{equation}\label{eq_5}
    V(x,y) = d(x) - d(y) - \langle \nabla d(y), x - y \rangle, \quad \forall x, y\in Q.
\end{equation}
As a result, we have $V(x, y) \geq \frac{1}{2} \|x - y\|^2.$
\end{definition}

Some standard examples of Bregman divergence are as follows:
\begin{enumerate}
    \item \textbf{Euclidean prox-function:}  Let $Q$ be a convex subset of $\mathbb{R}^n$ endowed with the Euclidean norm $\|\cdot \|_2$. Then, the \textit{Euclidean prox-function} on $Q$ is defined as $d(x) = \frac{1}{2} \|x\|_2^2 $ and the corresponding Bregman divergence is the standard square distance $V(x, y) = \frac{1}{2} \|x - y\|_2^2, \, \forall x, y \in Q. $

    \item \textbf{Entropic prox-function:} Let $Q = \Delta_n : = \left\{x \in \mathbb{R}^n_+:  \sum_{i=1}^nx_i=1 \right\}$ be the unit simplex in $\mathbb{R}^n$ endowed with the $L_1$-norm $\|\cdot \|_1$. Then, the \textit{entropic prox-function} (or \textit{Boltzmann-Shannon entropy}) on $Q$ is $d(x) = \sum_{i = 1}^{n} x_i \ln(x_i) $ and the corresponding Bregman divergence is the relative entropy $V(x, y) =\sum_{i = 1}^{n}  x_i \ln \left(x_i / y_i\right), \; \forall x \in Q, \forall y \in {\rm rint} Q. $

    If $d(x) = - \sum_{i = 1}^{n} x_i \ln(x_i) $. Then the corresponding Bregman divergence (called the generalized Kullback-Leibler (KL) divergence) is 
    \begin{equation}\label{KL_divergence}
        V(x,y) :=V_{\rm KL} (x, y) = \sum_{i = 1}^{n} \left( x_i \ln \left(\frac{x_i}{y_i}\right) - x_i - y_i \right). 
    \end{equation}

    \item \textbf{Log-barrier prox-function \cite{Chen1993Convergence}:} Let $Q = \mathbb{R}_{++}^n$ be the open positive orthant of $\mathbb{R}^n$. Then, the \textit{log-barrier prox-function} (or \textit{Burg's entropy}) on $Q$ is defined as $d(x) = - \sum_{i = 1}^{n} \ln (x_i)$. The corresponding Bregman divergence is known as the \textit{Itakura-Saito} (IS) divergence and is given by \begin{equation}\label{IS_divergence}
        V(x, y) :=V_{\rm IS} (x, y) = \sum_{i = 1}^{n} \left(\frac{x_i}{y_i} - \ln \left(\frac{x_i}{y_i}\right) - 1\right).
    \end{equation} 
\end{enumerate}

\begin{definition}\label{def_sm}
A function $f$ is called $L$-smooth relative to the prox-function $d$ on $Q$, if there exists $L > 0$ such that
\begin{equation}\label{eq_6}
    f(y) \leq f(x) + \langle \nabla f(x), y - x \rangle + L V(y,x),  \quad  \forall x \in {\rm rint}Q, y \in Q.
\end{equation}
\end{definition} 

The definition of relative smoothness in \eqref{eq_6} gives an upper approximation of $f$ similar to \eqref{eq_4}. Therefore, it is natural to consider a more general algorithm than \eqref{eq_PGM}, by replacing $\frac{1}{2}\|x- x_k\|^2$ in \eqref{eq_PGM} with Bregman divergence $V(x, x_k)$. Thus, by taking $x_0 \in {\rm rint} Q$ as an initial point, we get the following method (called \textit{Bregman proximal gradient} method) 
\begin{equation}\label{eq_relativePGM}
    x_{k+1} = \arg\min\limits_{x\in Q} \left\{f(x_k) + \left\langle \nabla f(x_k), x - x_k \right\rangle + L V(x,x_k)\right\}, \;\; \forall k \geq 0,
\end{equation}

Let us consider some examples of relatively smooth optimization problems.

\begin{example}[Poisson inverse problem \cite{hanzely2021accelerated}]\label{ex:InvPois}
Let $A \in \mathbb{R}^{m \times  n}_{+}$ be a non-negative observation matrix, and  $b \in \mathbb{R}^m_{++}$  be a noisy measurement vector. The task in this problem is to reconstruct a signal $x \in \mathbb{R}^n_+$ such that $Ax \approx b$. A natural measure of the closeness of two non-negative vectors is the KL-divergence defined in \eqref{KL_divergence}. Let us consider the following optimization problem
\begin{equation}\label{problem_invPoisson}
    \min_{x \in \mathbb{R}_+^n} \left\{ F(x) : = V_{\rm KL} (b, Ax)  + \psi (x)\right\}, 
\end{equation}
where $\psi$ is a simple regularization function. It is shown in \cite{Bauschke2017first} that the function $f(x) = V_{\rm KL} (b, Ax) $ is relatively $L$-smooth relative to $d(x) = - \sum_{i = 1}^{n} \ln(x_i)$ on $\mathbb{R}^n_+$ for any $L \geq \|b\|_1 = \sum_{i = 1}^{n} b_i$.  
\end{example}

\begin{example}[D-optimal design problem \cite{hanzely2021accelerated}]\label{ex:DOptDes}
Given $n$ vectors $v_1, v_2, \dots, v_n \in \mathbb{R}^m$ where $n \geq m + 1$. The problem is as follows
\begin{equation}\label{D-OptimProblem}
    \min_{x \in \Delta_n} \left\{f(x) := -\log \left(\det \left(\sum_{i=1}^nx_iv_iv_i^{\top} \right) \right)\right\},
\end{equation}
where $\Delta_n $ is a unit simplex in $\mathbb{R}^n$. By corresponding with the problem \eqref{main_composite_problem}, we find that $Q = \Delta_n$.    

The function $f$ defined in \eqref{D-OptimProblem}, is relatively $1$-smooth relative to the log-barrier prox-function $d(x) = -\sum_{i=1}^n \ln(x_i)$ on the set $\mathbb{R}^n_{++}$ \cite{Lu2018Relatively}. In this case, the Bregman divergence  is $V_{\rm IS}$, which defined in \eqref{IS_divergence}
\end{example}

\subsection{Inexact oracle}

In some works, such as \cite{inexact,stonyakin2021inexact} (for example only) one can find many first-order methods for solving different minimization problems $\min_{x \in Q} f(x)$, in many settings of the objective $f$, with so-called inexact oracle. 

Most minimization methods for such problems are constructed using some model of the objective $f$ at the current iterate $x_k$. This can be a quadratic model based on the $L$-smoothness of the objective
\begin{equation}\label{eq:modelFOM}
   f(x_k) + \langle \nabla f(x_k), x - x_k \rangle + \frac{L}{2} \|x - x_k\|^2.
\end{equation}

The literature on the first-order methods \cite{Devolder2014First,kamzolov2021universal} (for example only) also considers gradient methods with inexact information, relaxing the model \eqref{eq:modelFOM} to the following
\begin{equation}\label{eq:modelFOM_inexact}
   f_{\delta}(x_k) + \langle \nabla f_{\delta}(x_k), x - x_k \rangle + \frac{L}{2} \|x - x_k\|^2 + \delta, 
\end{equation}
where $(f_{\delta}, \nabla f_{\delta})$ (see Definition \ref{def1}) is called the inexact oracle, and this model represents an upper bound for the objective $f$. In particular, this relaxation allows obtaining universal gradient methods \cite{inexact_met}. 

\bigskip 

In what follows we mention the precise definition of the inexact oracle. 
\begin{definition}[\cite{stonyakin2021inexact}]\label{def1}
We say that the pair $(f_{\delta}(y), \nabla f_{\delta}(y))\in \mathbb{R}\times E^* $ is a $(\delta, L)$-oracle of the function $f$ at the point $y$ for $\delta,L>0$,  if it holds the following inequality
\begin{equation}\label{eq_8}
    0\leq f(x) - (f_{\delta}(y) + \left\langle\nabla f_{\delta}(y), x-y\right\rangle) \leq LV(x,y) + \delta, \quad \forall x \in Q.
\end{equation}
\end{definition}

\subsection{Triangular scaling property}

\begin{definition}\label{def1_1}
Let $d: Q \longrightarrow \mathbb{R}$ be a prox-function and $V(\cdot, \cdot)$ be a corresponding Bregman divergence. We say that $V(\cdot, \cdot)$ satisfies the \textit{triangular scaling property} \cite{hanzely2021accelerated}, with constant $\gamma > 0$ (which called the \textit{triangular scaling factor}),  if for any $x, z, \widetilde{z} \in Q$  it holds the following inequality
\begin{equation}\label{eq_15}
    V \left( (1 - \theta)x + \theta z, (1 - \theta)x + \theta \widetilde{z} \right) \leq \theta^{\gamma} V\left(z, \widetilde{z} \right),  \quad  \forall \theta \in [0, 1].
\end{equation}
\end{definition}

If $V(x,y)$ is jointly convex in $(x, y)$, then the inequality \eqref{eq_15} holds with $\gamma = 1$, because
\[
    V((1 - \theta)x + \theta z, (1 - \theta)x + \theta \widetilde{z}) \leq (1 - \theta)V(x,x) + \theta V( z,\widetilde{z} ) = \theta V(\widetilde{z}, z). 
\]

Next, we mention some examples of the Bregman divergence with the triangular scaling property. 

\begin{example}[Squared Euclidean distance]
Let $d(x) = \frac{1}{2}\| x \|^2_2$ and $V(x,y) = \frac{1}{2} \| x-y \|^2_2$. Indeed $V(x,y)$ is jointly convex in its two arguments. For this divergence, \eqref{eq_15} will be as the following
\[ 
    \frac{1}{2} \left\| (1-\theta)x + \theta z - ((1-\theta)x + \theta \widetilde{z})  \right \|^2_2 = \frac{1}{2} \left\| \theta (z - \widetilde{z})\right \|^2_2 =   \frac{\theta^2}{2} \left\|z - \widetilde{z} \right\|^2_2. 
\]
Therefore, the squared Euclidean distance satisfies the triangular scaling property with a scaling factor $\gamma = 2$. 
\end{example}

\begin{example}[Generalized Kullback-Leibler divergence]
Let $d(x) = - \sum_{i=1}^n x_{i} \ln (x_{i})$ defined on $\mathbb{R}^n_{+}$. The corresponding Bregman divergence is the KL divergence, defined in \eqref{KL_divergence}. 
Indeed $V_{\rm KL}(x,y)$ satisfies the triangular scaling property with a scaling factor $\gamma = 1$ (see \cite{hanzely2021accelerated}).    
\end{example}

\begin{example}[Itakuro-Saito divergence]
Let $d(x) = - \sum_{i=1}^n \ln (x_{i})$ defined on $\mathbb{R}^n_{++}$. The corresponding Bregman divergence is the IS divergence, defined in \eqref{IS_divergence}. Indeed $V_{\rm IS}(x,y)$ satisfies the triangular scaling property with a scaling factor $\gamma < 1$ (see \cite{hanzely2021accelerated}).     
\end{example}

\section{Fast gradient method for relatively smooth optimization problems with inexact oracle}\label{sect:FGM}
 
In this section, we assume that $V(\cdot, \cdot)$ satisfies the \textit{triangular scaling property} \cite{hanzely2021accelerated}, with constant $\gamma = 2$. We consider an adaptive fast gradient method (listed as Algorithm \ref{alg:cap} below) for solving the problem \eqref{main_composite_problem}, when $f$ is a relatively $L$-smooth function. It exploits the idea of an inexact oracle and the triangular scaling property of the Bregman divergence corresponding to the used prox structure. We mention here that this algorithm is an analog of the algorithm that was proposed in \cite{stonyakin2021inexact} based on the similar triangles method with an inexact model, which covers the idea of the inexact oracle as a special case. The proposed Algorithm \ref{alg:cap} is universal in the sense that it is applicable not only to relatively smooth but also to relatively Lipschitz continuous optimization problems.

\begin{algorithm}[htb]
\caption{Adaptive Fast Gradient Method with Inexact Oracle (AdapFGM).}\label{alg:cap}
\begin{algorithmic}[1]
\STATE \textbf{Input:} initial point $x_0$, $\{\delta_k\}_{k \ge 0}$ and $L_0 > 0$ such that $L_0 < 2L$. 
\STATE Set $y_0 := x_0, u_0 := x_0, \alpha_0 := 0, A_0:=\alpha_0$. 
\FOR{$k \ge 0$}
\STATE Find the smallest integer $i_k \ge 0$, such that
\begin{equation}\label{eq_10}
    f_{\delta_k}(x_{k + 1}) \leq f_{\delta_k}(y_{k + 1}) + \left\langle \nabla f_{\delta_k}(y_{k+1}),x_{k + 1}-y_{k + 1}\right\rangle + L_{k + 1}V(x_{k + 1}, y_{k+1}) + \delta_k,
\end{equation}
where $L_{k+1} = 2^{i_k - 1}L_k$, $\alpha_{k+1}$ is the largest root of the following equation
\begin{equation}\label{eq_11}
    A_{k+1}= L_{k+1}\alpha^2_{k+1}, \;\; \text{where } \;\; A_{k+1} := A_k + \alpha_{k+1}, 
\end{equation}
\begin{equation}\label{eq_12}
    y_{k+1} := \frac{\alpha_{k+1}u_k + A_kx_k}{A_{k+1}},
\end{equation}
\begin{equation}\label{eq_13}
    u_{k+1} := \arg\min\limits_{x\in Q} \left\{ \alpha_{k+1} \langle \nabla f_{\delta_k}(y_{k+1}), x - y_{k + 1}\rangle + V(x, u_k)\right\},
\end{equation}
\begin{equation}\label{eq_14}
    x_{k+1} := \frac{\alpha_{k+1}u_{k+1} + A_kx_k}{A_{k+1}}.
\end{equation}
\ENDFOR
\end{algorithmic}
\end{algorithm}

For Algorithm \ref{alg:cap}, we have the following result. 

\begin{theorem}\label{theo:FGM}
Assume that Bregman divergence $V(\cdot, \cdot)$ satisfies the triangular scaling property with scaling factor $\gamma = 2$. Let $R>0$, such that $V(x_*, x_0)\leq R^2$, where $x_*$ is a solution of the problem \eqref{main_composite_problem} closest to the initial point $x_0$. Then after $N \geq 1 $ iterations of the Algorithm \ref{alg:cap}, we have
\begin{equation}\label{eq_16}
    f(x_N) - f(x_*) 
    \leq \frac{8 L R^2}{(N+1)^2} + \frac{2\sum_{k=0}^{N-1}A_{k+1}\delta_k}{A_N},
\end{equation}
where $L>0$ is the relative smoothness parameter of the function $f$.
\end{theorem}

To prove Theorem \ref{theo:FGM}, we need the following lemmas.

\begin{lemma}[\cite{stonyakin2021inexact}]\label{def4}
Let $\varphi(x)$ be a convex function, and let us assume that for some $z\in Q$ the following inequality holds
\begin{equation*}
    y = \arg\min\limits_{x\in Q}\{ \varphi(x) + V(x,z)\}.
\end{equation*}
Then
\begin{equation*}
    \varphi(x) + V(x,z) \ge \varphi(y) + V(y,z) + V(x,y), \quad \forall x \in Q.
\end{equation*}
\end{lemma}

\begin{lemma}[\cite{stonyakin2021inexact}]\label{lem}
Let the sequence $\{\alpha_k\}_{k \geq 0}$ satisfy the conditions
\begin{equation*}
    \alpha_0=0, \quad A_k =\sum\limits_{i=0}^k \alpha_i,\quad A_k=L_k\alpha^2_k,
\end{equation*}
where $L_k\leq 2 L $ for any $k\ge 0$. Then it holds the following inequality
\begin{equation*}
    A_k \ge \dfrac{(k+1)^2}{8L}, \quad \forall  k \ge 1. 
\end{equation*}
\end{lemma}

Now, let us prove the following lemma. 
\begin{lemma}\label{def3}
For Algorithm \ref{alg:cap}, it holds the following inequality
\begin{equation}\label{eq:lemmafund}
    A_{k+1}f(x_{k+1}) - A_k f(x_k) + V(x, u_{k+1}) - V(x,u_k)\leq \alpha_{k+1} f(x) + 2\delta_k A_{k+1}, \; \forall x \in Q.
\end{equation}
\end{lemma}
\begin{proof}
Based on Definition \ref{def1}, where $x = y$, we have $f(x) - \delta \leq f_{\delta}(x) \leq f(x)$, and using \eqref{eq_10} we obtain
\begin{align*}
    f(x_{k+1})  & \leq f_{\delta_k}(x_{k+1}) + \delta_k 
    \\& \leq  f_{\delta_k}(y_{k + 1}) + \left\langle \nabla f_{\delta_k}(y_{k+1}),x_{k + 1}-y_{k + 1} \right\rangle + L_{k + 1}V(x_{k + 1}, y_{k+1}) + 2 \delta_k. 
\end{align*}

From \eqref{eq_12} and \eqref{eq_14}, we get
\begin{align}
    f(x_{k+1}) & \leq  f_{\delta_k}(y_{k + 1}) + \left\langle \nabla f_{\delta_k}( y_{k + 1}), \frac{\alpha_{k+1} u_{k+1} + A_k x_k}{A_{k+1}} - y_{k + 1} \right\rangle  \nonumber
    \\& \quad + L_{k + 1} V \left( \frac{\alpha_{k+1} u_{k+1} + A_kx_k}{A_{k+1}}, \frac{\alpha_{k+1}u_{k} + A_kx_k}{A_{k+1}}\right) + 2 \delta_k. \label{11ffgh}
\end{align}

From \eqref{eq_11} we get $\alpha_{k+1} = A_{k+1} - A_k$. Thus we have
\begin{align*}
    & \quad \; V \left( \frac{\alpha_{k+1}u_{k+1} + A_kx_k}{A_{k+1}},\frac{\alpha_{k+1}u_{k} + A_kx_k}{A_{k+1}} \right)
    \\& = V \left( \frac{A_{k+1} - A_k}{A_{k+1}}u_{k+1} + \frac{A_k}{A_{k+1}}x_k,\frac{A_{k+1} - A_k}{A_{k+1}}u_{k} + \frac{A_k}{A_{k+1}}x_k\right) 
    \\& = V \left( \left( 1 - \frac{A_k}{A_{k+1}}\right) u_{k+1} + \frac{A_k}{A_{k+1}}x_k, \left( 1 - \frac{A_k}{A_{k+1}}\right) u_{k} + \frac{A_k}{A_{k+1}}x_k \right) 
\end{align*}

Let $\theta := \left( 1 - \frac{A_k}{A_{k+1}}\right)$. By using the triangular scaling property \eqref{eq_15}, with $\gamma = 2$, we get
\begin{align}\label{re124}
    V \left( \frac{\alpha_{k+1}u_{k+1} + A_kx_k}{A_{k+1}},\frac{\alpha_{k+1}u_{k} + A_kx_k}{A_{k+1}} \right) & \leq \left( 1 - \frac{A_k}{A_{k+1}} \right)^2V(u_{k+1},u_k)  \nonumber
    \\& = \frac{\alpha^2_{k+1}}{A^2_{k+1}}V(u_{k+1},u_k). 
\end{align}

Therefore, by substituting \eqref{re124} in \eqref{11ffgh}, we find
\begin{align*}
    f(x_{k+1}) & \leq  f_{\delta_k}(y_{k + 1}) + \left\langle\nabla f_{\delta_k} (y_{k + 1}), \frac{\alpha_{k+1} u_{k+1} + A_k x_k}{A_{k+1}} - y_{k + 1} \right\rangle 
    \\& \quad + L_{k + 1}\frac{\alpha^2_{k+1}}{A^2_{k+1}} V(u_{k+1}, u_k) + 2 \delta_k.
\end{align*}

Using $A_{k+1} = A_k + \alpha_{k+1}$, and the convexity of the function $ x \longmapsto \left\langle\nabla f_{\delta_k}(y),x-y\right\rangle$, we obtain
\begin{align*}
    f(x_{k+1}) & \leq  \frac{A_k}{A_{k+1}} \left(f_{\delta_k}(y_{k + 1}) + \left \langle\nabla f_{\delta_k}(y_{k + 1}),x_k- y_{k+1} \right\rangle \right) 
    \\& \quad + \frac{\alpha_{k+1}}{A_{k+1}} \left(f_{\delta_k} (y_{k + 1}) + \left\langle \nabla f_{\delta_k}(y_{k + 1}), u_{k+1} - y_{k+1} \right\rangle\right)
    \\& \quad + L_{k + 1}\frac{\alpha^2_{k+1}}{A^2_{k+1}} V(u_{k+1}, u_k) + 2 \delta_k.
\end{align*}

But $A_{k+1} = L_{k+1} \alpha_{k+1}^2$ (see \eqref{eq_11}). Thus,  we have
\begin{align}\label{542fg5f}
    f(x_{k+1}) & \leq  \frac{A_k}{A_{k+1}} \left(f_{\delta_k}(y_{k + 1}) + \left\langle \nabla f_{\delta_k} (y_{k + 1}), x_k- y_{k+1} \right\rangle \right)  \nonumber
    \\& \quad + \frac{\alpha_{k+1}}{A_{k+1}} \left(f_{\delta_k}(y_{k + 1}) + \left\langle\nabla f_{\delta_k} (y_{k + 1}),u_{k+1}-y_{k+1}\right\rangle + \frac{1}{\alpha_{k+1}}V(u_{k+1}, u_k)\right) \nonumber
    \\& \quad   + 2\delta_k.
\end{align}

Therefore, from \eqref{eq_8} (for the function $F$), and \eqref{542fg5f} we obtain
\begin{align}
    f(x_{k+1}) & \leq  \frac{A_k}{A_{k+1}} f(x_{k}) + \frac{\alpha_{k+1}}{A_{k+1}} \Bigg(f_{\delta_k}(y_{k + 1}) + \left\langle\nabla f_{\delta_k} (y_{k + 1}), u_{k+1}-y_{k+1}\right\rangle \nonumber
    \\& \qquad \qquad\qquad \qquad\qquad \qquad + \frac{1}{\alpha_{k+1}}V(u_{k+1},u_k) \Bigg) + 2\delta_k. \label{eq_18}
\end{align}

Based on Lemma \ref{def4} for the optimization problem \eqref{eq_13} with 
\[
    \varphi(x) = \alpha_{k+1}\left\langle\nabla f_{\delta_k}(y_{k + 1}),x-y_{k+1}\right\rangle, \;\;  z = u_k, \;\;  \text{and} \;\; y = u_{k+1}, 
\]
we obtain
\begin{align*}
     & \quad \; \alpha_{k+1} \left \langle \nabla f_{\delta_k}(y_{k + 1}),u_{k+1} - y_{k+1} \right \rangle + V(u_{k+1} , u_k) + V(x, u_{k+1})
     \\& \leq \alpha_{k+1}\left\langle \nabla f_{\delta_k}(y_{k + 1}),x - y_{k+1}\right\rangle + V(x,u_k). 
\end{align*}
i.e., 
\begin{align}\label{eq_19}
      \left \langle \nabla f_{\delta_k}(y_{k + 1}),u_{k+1} - y_{k+1} \right \rangle& 
      \nonumber
      + \frac{1}{\alpha_{k+1}} V(u_{k+1} , u_k)  \leq  
     \left\langle \nabla f_{\delta_k}(y_{k + 1}),x - y_{k+1}\right\rangle 
     \\& +  \frac{1}{\alpha_{k+1}} V(x,u_k) - \frac{1}{\alpha_{k+1}} + V(x, u_{k+1}). 
\end{align}

Combining \eqref{eq_18} and \eqref{eq_19}, we obtain
\begin{align*}
    f(x_{k+1}) & \leq \frac{A_{k}}{A_{k+1}} f(x_k) + \frac{\alpha_{k+1}}{A_{k+1}} \Big( f_{\delta_k}(y_{k+1}) + \left\langle\nabla f_{\delta_k}(y_{k + 1}),x-y_{k+1}\right\rangle
    \\& \quad + \frac{1}{\alpha_{k+1}}V(x,u_k) - \frac{1}{\alpha_{k+1}} V(x,u_{k+1})\Big) + 2\delta_k
    \\& = \frac{A_{k}}{A_{k+1}} f(x_k) + \frac{\alpha_{k+1}}{A_{k+1}} f(x)+ \frac{1}{A_{k+1}} V(x,u_k) - \frac{1}{A_{k+1}} V(x,u_{k+1}) + 2\delta_k. 
\end{align*}

From the last inequality, we get the following
\begin{equation*}
    A_{k+1}f(x_{k+1}) - A_k f(x_k) + V(x, u_{k+1}) - V(x,u_k)\leq \alpha_{k+1} f(x) + 2\delta_k A_{k+1}.
\end{equation*}
which is the desired inequality \eqref{eq:lemmafund}.
\end{proof}

Now, by using the previous Lemmas, we can prove Theorem \ref{theo:FGM}, as follows. 

\begin{proof}[Proof of Theorem \ref{theo:FGM}.]
By taking the sum of both sides of \eqref{eq:lemmafund}, over $k$ from 0 to $N - 1$, and set $x = x_{*}$, we get the following
\begin{equation}\label{eq_20}
    A_N f(x_N) \leq A_N f(x_*) + V(x_*, u_0) - V(x_*, u_N) + 2 \sum\limits_{k=0}^{N - 1} A_{k+1} \delta_k .
\end{equation}

But $V(x_*,u_{N}) \ge 0,$ and $u_0 = x_0$, thus we get
\begin{equation*}
    A_N f(x_N) - A_N f(x_*) \leq V(x_*, x_0) + 2\sum\limits_{k=0}^{N - 1} A_{k+1} \delta_k \leq R^2 + 2\sum\limits_{k=0}^{N - 1} A_{k+1} \delta_k.
\end{equation*}

Now, by dividing both sides of the last inequality by $A_N$, and since $A_N \geq \frac{(N+1)^2}{8L}$ (see Lemma \ref{lem}), we get the desired inequality \eqref{eq_16}.  
\end{proof}

\begin{remark}
For problem \eqref{main_composite_problem}, we will show that the Algorithm \ref{alg:cap} is universal, namely, we will prove that this method is applicable not only to relatively smooth but also to relatively Lipschitz continuous optimization problems. 

In \cite{adapt_alg} it is shown that if the function $f$ is relatively Lipschitz continuous, then for any fixed $L>0$ and $\delta>0$ the following inequality holds
\begin{equation*}
    f(x)\leq f(y)+\langle \nabla f(y), x-y\rangle + L V(x,y)+L V(y,x)+\delta, \quad \forall x,y\in Q.
\end{equation*}

Therefore, for $L_{k+1}\geq L$, taking into account the triangular scaling property for Bregman divergence, we obtain
\begin{align}
    f(x_{k+1}) & \leq f(y_{k+1})+\langle \nabla f(y_{k+1}), x_{k+1}-y_{k+1}\rangle + L_{k+1} V(x_{k+1},y_{k+1}) \nonumber
    \\& \quad +L_{k+1} V(y_{k+1},x_{k+1})+\delta \nonumber
    \\& = f(y_{k+1}) + \left\langle \nabla  f(y_{k+1}), \dfrac{\alpha_{k+1} u_{k+1} + A_k x_k}{A_{k+1}}-\dfrac{\alpha_{k+1} u_{k} + A_k x_k}{A_{k+1}} \right\rangle  \nonumber
    \\& \quad + L_{k+1} V\left(\dfrac{\alpha_{k+1} u_{k+1} + A_k x_k}{A_{k+1}}, \dfrac{\alpha_{k+1} u_{k} + A_k x_k}{A_{k+1}}\right)  \nonumber 
    \\& \quad + L_{k+1} V\left(\dfrac{\alpha_{k+1} u_{k} + A_k x_k}{A_{k+1}}, \dfrac{\alpha_{k+1} u_{k+1} + A_k x_k}{A_{k+1}}\right)+\delta  \nonumber
    \\& \leq f(y_{k+1})+\dfrac{\alpha_{k+1}}{A_{k+1}}\langle \nabla f(y_{k+1}), u_{k+1}-u_{k}\rangle  \nonumber
    \\& \quad +L_{k+1} V\left(\left(1-\dfrac{A_{k}}{A_{k+1}}\right) u_{k+1}+\dfrac{A_k}{A_{k+1}}x_k,\left(1-\dfrac{A_{k}}{A_{k+1}}\right)u_{k}+\dfrac{A_k}{A_{k+1}} x_k \right)     \nonumber
    \\& \quad + L_{k+1} V \left(\left(1-\dfrac{A_{k}}{A_{k+1}}\right)u_{k}+\dfrac{A_k}{A_{k+1}}x_k,\left(1-\dfrac{A_{k}}{A_{k+1}}\right)u_{k+1}+\dfrac{A_k}{A_{k+1}}x_k\right)+\delta       \nonumber
    \\& \leq f(y_{k+1})+\dfrac{\alpha_{k+1}}{A_{k+1}}\langle \nabla f(y_{k+1}), u_{k+1} - u_{k}\rangle     \nonumber
    \\& \quad +\dfrac{L_{k+1} \alpha^2_{k+1}}{A^2_{k+1}} \left( V(u_{k+1}, u_k) + V(u_k, u_{k+1})\right) + \delta    \nonumber
    \\& = f(y_{k+1}) + \dfrac{1}{A_{k+1}} \left(\alpha_{k+1} \langle \nabla f(y_{k+1}), u_{k+1} - u_{k}\rangle+ V(u_{k+1},u_k) + V(u_k,u_{k+1})\right)      \nonumber
    \\& \quad +\delta. \label{1111}
\end{align}

Let us consider as an $(\delta, L)$-oracle of the function $f$ the pair $(f_\delta (y), \nabla f_\delta (y)) = (f(y), \nabla f(y))$. Since $u_{k+1}:=\arg\min\limits_{x\in Q} \left\{\alpha_{k+1}\langle \nabla f(y_{k+1}), x-y_{k+1}\rangle + V(x,u_k)\right\}$, then by Lemma \ref{def4}, it holds the following inequality
\begin{equation*}
    \alpha_{k+1}\langle \nabla f(y_{k+1}), u_{k+1}-x\rangle \leq V(x,u_k) - V(u_{k+1},u_k) - V(x,u_{k+1}).
\end{equation*}
Hence, for $x=u_k$, we have
\begin{equation*}
    \alpha_{k+1}\langle \nabla f(y_{k+1}), u_{k+1} -u_k\rangle \leq - V(u_{k+1}, u_k)-V(u_k,u_{k+1}).
\end{equation*}
Therefore, from \eqref{1111} we get
\begin{equation*}
    f(x_{k+1})\leq f(y_{k+1})+\delta.
\end{equation*}

Since for relatively Lipschitz continuous functions the following inequality also holds
\begin{equation*}
    \langle \nabla f(y), x - y \rangle + L V(x, y) + \delta \geq 0, \quad \forall x, y \in Q.
\end{equation*}
Then
\begin{equation*}
    f(x_{k+1}) \leq f(y_{k+1}) + \delta \leq f(y_{k+1}) + \langle \nabla f(y_{k+1}), x_{k+1} -y_{k+1} \rangle + L V(x_{k+1}, y_{k+1}) + 2 \delta.
\end{equation*}

Thus, the exit criterion from the iteration of the Algorithm \ref{alg:cap} will be guaranteed to be satisfied when $L_{k+1}\geq L$ and $\delta_{k+1}\geq 2\delta$.
\end{remark}

\section{Accelerated Bregman proximal gradient  methods with inexact oracle}\label{sect:auccBPGM}

In this section, to solve the problem \eqref{main_composite_problem}, we propose two accelerated algorithms. The first one is non-adaptive (see Algorithm \ref{alg:cap_1}) and it is an analog of the proposed method APBG in \cite{hanzely2021accelerated}, with an adaptation to the inexact oracle. The second algorithm (see Algorithm \ref{alg:cap_2}) is an adaptive version of Algorithm \ref{alg:cap_1}, where the adaptivity is for the relative smoothness parameter $L$. 

At the first, let us introduce the following notation
\begin{equation*}
    g(x|y) := f_{\delta}(y) + \left\langle\nabla f_{\delta}(y),x-y\right\rangle, \quad \forall x \in Q, \;\; \text{and} \;\; y \in {\rm rint Q}.
\end{equation*}

By taking into account that $f$ is a relatively $L$-smooth function and from \eqref{eq_8} we obtain that
\begin{equation} \label{eq_22}
    g(x|y)\leq f(x) \leq g(x|y) + L V(x,y) + \delta .
\end{equation}

\subsection{Non-adaptive Accelerated Bregman Proximal Gradient Method with Inexact Oracle}

In this subsection, we propose an analog of the non-adaptive accelerated Bregman proximal gradient method (APBG) \cite{hanzely2021accelerated}. By adding the idea of an inexact oracle to APBG we obtain a new algorithm, listed as Algorithm \ref{alg:cap_1}, below. 

\begin{algorithm}[htb]
\caption{Accelerated Bregman Proximal Gradient Method with Inexact Oracle  (non-adaptive version, AccBPGM-1). }\label{alg:cap_1}
\begin{algorithmic}[1]
\STATE\textbf{Input:} initial point $x_0\in {\rm rint} Q$, $\gamma \in (1,2], \delta >0$, and $L > 0$.
\STATE Set $z_0 = x_0$ and $\theta_0 = 1$.
\FOR{$k = 0, 1, 2, \ldots$}
\STATE $y_k = (1 - \theta_k)x_k + \theta_kz_k.$
\STATE $z_{k+1} = \arg\min\limits_{z \in Q} \left\{ g(z|y_k) + \theta_k^{\gamma-1}LV(z,z_k) \right\}.$
\STATE $x_{k+1} = (1 - \theta_k)x_k + \theta_kz_{k+1}.$
\STATE Choose $\theta_{k+1} \in (0, 1]$ such that
\begin{equation}\label{eq_21}
    \frac{1-\theta_{k+1}}{\theta_{k+1}^{\gamma}} \leq \frac{1}{\theta_k^{\gamma}}.
\end{equation}
\ENDFOR
\end{algorithmic}
\end{algorithm}

For Algorithm \ref{alg:cap_1}, we have the following result. 

\begin{theorem}\label{def70}
Let $f : Q \longrightarrow \mathbb{R}$ be an $L$-smooth function relative to a prox-function $d$ on $Q$, and its Bregman divergence satisfies the triangular scaling property with scaling factor $\gamma \in (1,2]$. After $N \geq 1$ iterations of the Algorithm \ref{alg:cap_1}, if $\theta_N \leq \frac{\gamma}{N + \gamma}$, then it holds the following inequality
\begin{equation}\label{ineq:rate_alg2}
    f(x_N) - f(x_*) \leq \Big{(} \frac{\gamma}{\gamma + N-1} \Big{)}^{\gamma}LV(x_*,x_0) + \delta N.
\end{equation}
\end{theorem}

To prove Theorem \ref{def70}, we need the following lemmas.

\begin{lemma}[\cite{hanzely2021accelerated}]\label{def60}
For any $\gamma \in (1,2]$. The sequence $\left\{\theta_k := \frac{\gamma}{k + \gamma}\right\}_{k \geq 0}$ satisfies the condition \eqref{eq_21}.
\end{lemma}

\begin{lemma}\label{def500}
Let $f : Q \longrightarrow \mathbb{R}$ be an $L$-smooth function relative to a prox-function $d$ on $Q$, and its Bregman divergence $V(\cdot, \cdot)$ satisfies the triangular scaling property with scaling factor $\gamma \in (1,2]$. Then, the sequences $\{\theta_k\}_{k \geq 0}, \{x_k\}_{k \geq 0}, \{z_k\}_{k \geq 0}$ generated by Algorithm \ref{alg:cap_1} satisfy the following inequality for any $x \in Q$
\begin{equation}\label{eq_23}
    \frac{1}{\theta_k^{\gamma}}\Big{(} f(x_{k+1}) - f(x) \Big{)} + LV(x,z_{k+1}) \leq \frac{1 - \theta_k}{\theta_k^{\gamma}}\Big{(} f(x_{k}) - f(x) \Big{)} + LV(x,z_k) + \frac{\delta}{\theta_k^{\gamma}}.
\end{equation}
\end{lemma}

\begin{proof}
From the right-hand side of \eqref{eq_22}, and from items 4, 6 of Algorithm \ref{alg:cap_1}, we find
\begin{align}
    f(x_{k+1}) & \leq g(x_{k+1}|y_k) + LV(x_{k+1},y_k) + \delta \nonumber 
    \\& = g(x_{k+1}|y_k) + LV((1 - \theta_k)x_k + \theta_kz_{k+1},(1 - \theta_k)x_k + \theta_kz_k) + \delta \nonumber
    \\& \overset{\eqref{eq_15}}{\leq} g(x_{k+1}|y_k) + \theta_k^{\gamma}LV(z_{k+1},z_k) + \delta. \label{eq_24}
\end{align}

Using $x_{k+1} = (1 - \theta_k)x_k + \theta_kz_{k+1}$, and from the convexity of the function $x \longmapsto g(x|y_k) \; \forall x \in Q$, we get
\begin{align}
    f(x_{k+1}) & \leq (1 - \theta_k)g(x_{k}|y_k) + \theta_k g(z_{k+1}|y_k) + \theta_k^{\gamma}LV(z_{k+1},z_k) + \delta \nonumber 
    \\& = (1 - \theta_k)g(x_{k}|y_k) + \theta_k(g(z_{k+1}|y_k)+ \theta_k^{\gamma-1}LV(z_{k+1},z_k)) + \delta. \label{eq_25}
\end{align}

Applying Lemma \ref{def4} for $\varphi(x) = \frac{1}{\theta_k^{\gamma-1}L} g(x|y_k)$, and since 
\[
    z_{k+1} = \arg\min\limits_{z \in Q} \left\{ g(z|y_k) + \theta_k^{\gamma-1}LV(z,z_k) \right\}
\]
(see item 5 of Algorithm \ref{alg:cap_1}) we obtain, for any $x \in Q$, the following inequality
\begin{equation}\label{eq_jgfgk}
    g(z_{k+1}|y_k) + \theta_k^{\gamma-1}LV(z_{k+1},z_k) \leq g(x|y_k) + \theta_k^{\gamma-1}LV(x,z_k) - \theta_k^{\gamma-1}LV(x,z_{k+1}).
\end{equation}
Hence, from \eqref{eq_25} and \eqref{eq_jgfgk}, for any $x \in Q$, we find 
\begin{align*}
    f(x_{k+1}) & \leq (1 - \theta_k)g(x_k|y_k) + \theta_k(g(x|y_k) +\theta_k^{\gamma-1}LV(x,z_k) - \theta_k^{\gamma-1}LV(x,z_{k+1})) + \delta
    \\& = (1 - \theta_k)g(x_k|y_k) + \theta_kg(x|y_k) +\theta_k^{\gamma}(LV(x,z_k) - LV(x,z_{k+1})) + \delta
    \\& \overset{\eqref{eq_22}}{\leq} (1 - \theta_k)f(x_k) + \theta_kf(x) +\theta_k^{\gamma}(LV(x,z_k) - LV(x,z_{k+1})) + \delta.
\end{align*}

Subtracting $f(x)$ from both sides of the last inequality, we obtain
\begin{equation*}
    f(x_{k+1}) - f(x) \leq (1 - \theta_k)(f(x_k) - f(x)) + \theta_k^{\gamma}(LV(x,z_k) - LV(x,z_{k+1})) + \delta.
\end{equation*}

Dividing both sides of the last inequality by $\theta_k^{\gamma}$, we get the desired inequality \eqref{eq_23}. 
\end{proof}

Now, by using the Lemma \ref{def60} and Lemma \ref{def500}, we can prove Theorem \ref{def70}, as follows. 

\begin{proof}[Proof of Theorem \ref{def70}]
By taking the sum of both sides of \eqref{eq_23} for $x=x_*$, over $k$ from $0$ to $N-1$, and taking into account \eqref{eq_21}, we get the following
\begin{align*}
    & \quad \sum_{k=0}^{N-1}\left(\frac{1}{\theta_{k}^{\gamma}}\left( f(x_{k+1}) - f(x_*) \right) + LV(x_*,z_{k+1}) \right) 
    \\& \leq \sum_{k=0}^{N-1}\left( \frac{1 - \theta_k}{\theta_k^{\gamma}}\left( f(x_{k}) - f(x_*) \right)+ LV(x_*,z_k) + \frac{\delta}{\theta_k^{\gamma}}\right)
    \\& =\frac{1 - \theta_0}{\theta_0^{\gamma}}\left( f(x_{0}) - f(x_*) \right)+\sum_{k=1}^{N-1}\left( \frac{1 - \theta_k}{\theta_k^{\gamma}}\left( f(x_{k}) - f(x_*) \right)\right)+ \sum_{k=0}^{N-1}LV(x_*,z_k)
    \\& \quad + \sum_{k=0}^{N-1}\frac{\delta}{\theta_k^{\gamma}}
    \\&\leq \frac{1 - \theta_0}{\theta_0^{\gamma}}\left( f(x_{0}) - f(x_*) \right)+\sum_{k=1}^{N-1}\left( \frac{1}{\theta_{k-1}^{\gamma}}\left( f(x_{k}) - f(x_*) \right)\right)+ \sum_{k=0}^{N-1}LV(x_*,z_k) 
    \\& \quad + \sum_{k=0}^{N-1}\frac{\delta}{\theta_k^{\gamma}}.
\end{align*}

Then, after taking the summation we have
\begin{equation*}
    \frac{1}{\theta_{N-1}^{\gamma}} \left( f(x_N) - f(x_*) \right)\leq \frac{1 - \theta_0}{\theta_0^{\gamma}} \left( f(x_0) - f(x_*) \right) + LV(x_*,z_0) - LV(x_*,z_N) + \sum_{k=0}^{N-1} \frac{\delta}{\theta_k^{\gamma}}.
\end{equation*}

Based on the fact that $V(x_*,z_N) \ge 0$, and also $\frac{1 - \theta_0}{\theta_0^{\gamma}} = 0$, since $\theta_0 = 1$, we have
\begin{equation*}
    \frac{1}{\theta_{N-1}^{\gamma}}\left( f(x_N) - f(x_*) \right)  \leq LV(x_*,z_0) + \sum_{k=0}^{N-1}\frac{\delta}{\theta_k^{\gamma}}.
\end{equation*}

By multiplying both sides of the last inequality by $\theta_{N-1}^{\gamma}$ and taking into account that $z_0 = x_0$, we have
\begin{equation}\label{eq:pohn1}
    f(x_N) - f(x_*)   \leq L\theta_{N-1}^{\gamma} V(x_*,x_0) + \sum_{k=0}^{N-1}\frac{\theta_{N-1}^{\gamma}}{\theta_k^{\gamma}}\delta.
\end{equation}

Now, by taking $\theta_k = \frac{\gamma}{k + \gamma}, \forall k \geq 0$, then from Lemma \ref{def60}, we obtain
\begin{equation*}
    \frac{\theta_{N-1}}{\theta_{k}} = \frac{\gamma}{N-1 + \gamma} \cdot \frac{k + \gamma}{\gamma} = \frac{k+\gamma}{N-1+\gamma}\leq 1,\quad \forall k=0,1,\dots, N-1.
\end{equation*}
Thus,
\begin{equation}\label{eq:14gf52}
    \sum_{k=0}^{N-1}\frac{\theta_{N-1}^{\gamma}}{\theta_k^{\gamma}}\delta = \delta \left( \frac{\theta_{N-1}^{\gamma}}{\theta_0^{\gamma}} + \frac{\theta_{N-1}^{\gamma}}{\theta_1^{\gamma}} + \ldots + \frac{\theta_{N-1}^{\gamma}}{\theta_{N-1}^{\gamma}}  \right) \leq \delta N.
\end{equation}

Therefore, from \eqref{eq:pohn1} and \eqref{eq:14gf52} we get the desired inequality \eqref{ineq:rate_alg2}.
\end{proof}

\subsection{Adaptive Accelerated Proximal Gradient Bregman Method with Inexact Oracle}

In this subsection, we propose an adaptive version of Algorithm \ref{alg:cap_1}, where the adaptivity is for the relative smoothness parameter $L$. This adaptive algorithm is listed as Algorithm \ref{alg:cap_2}.

\begin{algorithm}[htp]
\caption{Accelerated Bregman Proximal Gradient Method with Inexact Oracle (adaptive version, AccBPGM-2).} \label{alg:cap_2}
\begin{algorithmic}[1]
\STATE \textbf{Input:}  initial point $x_0\in {\rm rint} Q$, $\gamma \in (1,2], \delta >0$, and $L_1 >0$ such that $L_1 < 2L$.
\STATE Set $z_0 = x_0$, $\theta_0 = 1$.
\FOR{$k = 0, 1, 2, \ldots$}
\STATE $y_k = (1 - \theta_k) x_k + \theta_k z_k. $
\STATE $z_{k+1} = \arg\min\limits_{z \in Q} \left\{ g(z|y_k) + \theta_k^{\gamma-1}L_{k+1}V(z,z_k)\right\}. $
\STATE $x_{k+1} = (1 - \theta_k)x_k + \theta_kz_{k+1} .$
\STATE Choose $\theta_{k+1} \in (0, 1]$ such that
\begin{equation}\label{eq_26}
    \frac{1-\theta_{k+1}}{\theta_{k+1}^{\gamma}} \leq \frac{1}{\theta_k^{\gamma}}.
\end{equation}
\STATE Find $L_{k+1} \ge  \left(\frac{\theta_{k+1}}{\theta_k}\right)^{\gamma}$ and $L_{k+1}\geq L_k$ such that
\begin{equation*}
    f_{\delta}(x_{k+1}) \leq f_{\delta}(y_k) + \left\langle\nabla f_{\delta}(y_k),x_{k+1}-y_k\right\rangle + L_{k+1}V(x_{k+1},y_k) + \delta.
\end{equation*}
\ENDFOR
\end{algorithmic}
\end{algorithm}

To analyze Algorithm \ref{alg:cap_2} and obtain its convergence rate, at first, we will prove the following lemma.  

\begin{lemma}\label{def8}
Let $f: Q \longrightarrow \mathbb{R}$ be an $L$-smooth function relative to a prox-function $d$ on $Q$, and its Bregman divergence $V(\cdot, \cdot)$ satisfies the triangular scaling property with scaling factor $\gamma \in (1,2]$. Then, the sequences $\{\theta_k\}_{k \geq 0}, \{x_k\}_{k \geq 0}, \{z_k\}_{k \geq 0}$ generated by Algorithm \ref{alg:cap_2} satisfy the following inequality for any $x \in Q$
\begin{equation}\label{eq_27}
    \frac{1}{\theta_k^{\gamma}}\Big{(} f(x_{k+1}) - f(x) \Big{)} + L_{k+1}V(x,z_{k+1}) \leq \frac{1 - \theta_k}{\theta_k^{\gamma}}\Big{(} f(x_{k}) - f(x) \Big{)} + L_{k+1}V(x,z_k) + \frac{\delta}{\theta_k^{\gamma}}.
\end{equation}
\end{lemma}

\begin{proof} 
From the right-hand side of \eqref{eq_22}, and from items 4, 6 of Algorithm \ref{alg:cap_2}, we find
\begin{align}
    f(x_{k+1}) & \leq g(x_{k+1}|y_k) + L_{k+1}V(x_{k+1},y_k) + \delta \nonumber 
    \\& = g(x_{k+1}|y_k) + L_{k+1}V((1 - \theta_k)x_k + \theta_kz_{k+1},(1 - \theta_k)x_k + \theta_kz_k) + \delta \nonumber
    \\& \overset{\eqref{eq_15}}{\leq} g(x_{k+1}|y_k) + \theta_k^{\gamma} L_{k+1} V(z_{k+1},z_k) + \delta. \label{eq_28}
\end{align}

Using $x_{k+1} = (1 - \theta_k)x_k + \theta_kz_{k+1}$, and from the convexity of the function $ x \longmapsto g( x |y_k)\; \forall x \in Q$, we get
\begin{align}
    f(x_{k+1}) & \leq (1 - \theta_k)g(x_{k}|y_k) + \theta_k g(z_{k+1}|y_k) + \theta_k^{\gamma} L_{k+1} V(z_{k+1},z_k) + \delta \nonumber 
    \\& = (1 - \theta_k)g(x_{k}|y_k) + \theta_k(g(z_{k+1}|y_k)+ \theta_k^{\gamma-1} L_{k+1} V(z_{k+1},z_k)) + \delta. \label{eq_29}
\end{align}

Applying Lemma \ref{def4} for 
 $\varphi(x) = \frac{1}{\theta_k^{\gamma-1}L_{k+1}} g(x|y_k)$, and since 
\[
    z_{k+1} = \arg\min\limits_{z \in Q} \left\{ g(z|y_k) + \theta_k^{\gamma-1} L_{k+1} V(z,z_k) \right\}
\]
(see item 5 of Algorithm \ref{alg:cap_2}) we obtain, for any $x \in Q$, the following inequality
\begin{equation}\label{eq_jgfgkfd}
    g(z_{k+1}|y_k) + \theta_k^{\gamma-1}L_{k+1} V(z_{k+1},z_k) \leq g(x|y_k) + \theta_k^{\gamma-1} L_{k+1} V(x,z_k) - \theta_k^{\gamma-1} L_{k+1} V(x,z_{k+1}).
\end{equation}
Hence, from \eqref{eq_29} and \eqref{eq_jgfgkfd}, for any $x \in Q$, we find 
\begin{align*}
    f(x_{k+1}) & \leq (1 - \theta_k) g(x_k|y_k) + \theta_k \left( g(x|y_k) +\theta_k^{\gamma-1}L_{k+1}V(x,z_k) - \theta_k^{\gamma-1}L_{k+1}V(x,z_{k+1}) \right) 
    \\& \quad + \delta
    \\& = (1 - \theta_k)g(x_k|y_k) + \theta_k g(x|y_k) +\theta_k^{\gamma} \left(L_{k+1}V(x,z_k) - L_{k+1}V(x,z_{k+1})\right) + \delta 
    \\& \overset{\eqref{eq_22} }{\leq} (1 - \theta_k)f(x_k) + \theta_kf(x) +\theta_k^{\gamma} \left(L_{k+1}V(x,z_k) - L_{k+1}V(x,z_{k+1}) \right) + \delta.
\end{align*}

Subtracting $f(x)$ from both sides of the last inequality, we obtain
\begin{equation*}
    f(x_{k+1}) - f(x) \leq (1 - \theta_k)(f(x_k) - f(x)) + \theta_k^{\gamma}(L_{k+1}V(x,z_k) - L_{k+1}V(x,z_{k+1})) + \delta.
\end{equation*}

Dividing both sides of the last inequality by $\theta_k^{\gamma}$, we get the desired inequality \eqref{eq_27}. 
\end{proof}

For Algorithm \ref{alg:cap_2}, we have the following result. 

\begin{theorem}\label{def10}
Let $f: Q \longrightarrow \mathbb{R}$ be an $L$-smooth function relative to a prox-function $d$ on $Q$, and the Bregman divergence $V(\cdot, \cdot)$ satisfies the triangular scaling property with scaling factor $\gamma \in (1,2]$. After $N \geq 1$ iterations of the Algorithm \ref{alg:cap_2}, if $\theta_N \leq \frac{\gamma}{N + \gamma}$, then it holds the following inequality
\begin{equation}\label{ineq_rate_alg3}
    f(x_N) - f(x_*) \leq 2L\left( \frac{\gamma}{\gamma + N-1} \right)^{\gamma}V(x_*,x_0) + \left(2(N-1)L+1\right)\delta.
\end{equation}
\end{theorem}

\begin{proof} 
By taking the sum of both sides of \eqref{eq_27} for $x=x_*$, over $k$ from $0$ to $N-1$ (having previously divided both parts by $L_{k+1}$), and taking into account \eqref{eq_26} and condition $L_{k+1}\geq L_k$, we get the following inequality
\begin{align*}
    & \quad \sum_{k=0}^{N-1}\left(\frac{1}{\theta_{k}^{\gamma}L_{k+1}}\left( f(x_{k+1}) - f(x_*) \right) + V(x_*,z_{k+1}) \right) 
    \\& \leq \sum_{k=0}^{N-1}\left( \frac{1 - \theta_k}{\theta_k^{\gamma}L_{k+1}}\left( f(x_{k}) - f(x_*) \right)+ V(x_*,z_k) + \frac{\delta}{\theta_k^{\gamma}L_{k+1}}\right)
    \\& =\frac{1 - \theta_0}{\theta_0^{\gamma}L_1}\left( f(x_{0}) - f(x_*) \right)+\sum_{k=1}^{N-1}\left( \frac{1 - \theta_k}{\theta_k^{\gamma}L_{k+1}}\left( f(x_{k}) - f(x_*) \right)\right)+ \sum_{k=0}^{N-1}V(x_*,z_k) 
    \\& \quad + \sum_{k=0}^{N-1}\frac{\delta}{\theta_k^{\gamma}L_{k+1}}
    \\&\leq \frac{1 - \theta_0}{\theta_0^{\gamma}L_1}\left( f(x_{0}) - f(x_*) \right)+\sum_{k=1}^{N-1}\left( \frac{1}{\theta_{k-1}^{\gamma}L_k}\left( f(x_{k}) - f(x_*) \right)\right)+ \sum_{k=0}^{N-1}V(x_*,z_k)
    \\& \quad + \sum_{k=0}^{N-1}\frac{\delta}{\theta_k^{\gamma}L_{k+1}}.
\end{align*}


Then, after taking the summation we have
\begin{align*}
    \frac{1}{\theta_{N-1}^{\gamma}L_N} \left( f(x_N) - f(x_*) \right) & \leq \frac{1 - \theta_0}{\theta_0^{\gamma}L_1} \left( f(x_0) - f(x_*) \right) + V(x_*,z_0)-V(x_*,z_N)  \\& \quad + \sum_{k=0}^{N-1} \frac{\delta}{\theta_k^{\gamma}L_{k+1}}.
\end{align*}

Based on the fact that $V(x_*,z_N) \ge 0,$, and also $\frac{1 - \theta_0}{\theta_0^{\gamma}} = 0$, since $\theta_0 = 1$, we have
\begin{equation*}
    \frac{1}{\theta_{N-1}^{\gamma}L_N}\left( f(x_N) - f(x_*) \right) \leq V(x_*,z_0) + \sum_{k=0}^{N-1}\frac{\delta}{\theta_k^{\gamma}L_{k+1}}.
\end{equation*}

By multiplying  both sides of the last inequality by $\theta_{N-1}^{\gamma}L_N$ and taking into account that $z_0 = x_0$, we have
\begin{equation}\label{eq:pol1240}
    f(x_N) - f(x_*)   \leq L_N\theta_{N-1}^{\gamma} V(x_*,z_0) + \sum_{k=0}^{N-1}\frac{\theta_{N-1}^{\gamma}L_N}{\theta_k^{\gamma}L_{k+1}}\delta.
\end{equation}


From Lemma \ref{def60} and the condition on $L_{k+1}$ for any $k=0,1,\dots,N-2$, we have
\begin{equation*}
    \frac{\theta_{N-1}^{\gamma}L_N}{\theta_{k}^{\gamma}L_{k+1}} \leq \frac{\theta_{N-1}^{\gamma}L_N}{\theta_{k+1}^{\gamma}} = \left( \frac{\gamma}{N-1 + \gamma} \right)^{\gamma} \cdot \left(\frac{k + 1 + \gamma}{\gamma}\right)^{\gamma}L_N = \left(\frac{k +1+\gamma}{N-1+\gamma} \right)^{\gamma}L_N \leq L_N. 
\end{equation*}

For $k=N-1$ we have $\frac{\theta_{N-1}^{\gamma}L_N}{\theta_{N-1}^{\gamma}L_{N}}=1.$
Thus,
\begin{equation}\label{eq:ojg102}
    \sum_{k=0}^{N-1}\frac{\theta_{N-1}^{\gamma}L_N}{\theta_k^{\gamma}L_{k+1}}\delta = \delta \left( \frac{\theta_{N-1}^{\gamma}L_N}{\theta_0^{\gamma}L_{1}} + \frac{\theta_{N-1}^{\gamma}L_N}{\theta_1^{\gamma}L_2} + ... + \frac{\theta_{N-1}^{\gamma}L_N}{\theta_{N-1}^{\gamma}L_{N}} \right) \leq \left((N-1)L_N+1\right)\delta.
\end{equation}

From \eqref{eq:pol1240} and \eqref{eq:ojg102}, we get the following inequality
\begin{equation*}
    f(x_N) - f(x_*) \leq \left( \frac{\gamma}{\gamma + N-1} \right)^{\gamma} L_N V(x_*,x_0) + \left((N-1)L_N+1\right)\delta.
\end{equation*}

In view of the assumption $L_1 < 2L$ and inequality \eqref{eq_22}, we obtain $L_{N} < 2L$, and the following inequality
\begin{equation*}
    f(x_N) - f(x_*) \leq 2L\left( \frac{\gamma}{\gamma + N-1} \right)^{\gamma}V(x_*,x_0) + \left(2(N-1)L+1\right)\delta,
\end{equation*}
which is the desired inequality \eqref{ineq_rate_alg3}.
\end{proof} 

Therefore, as a result, we have proposed an approach that allows tuning to the parameter $L_{k+1}$. It is expected that the number of calls to point 5 of Algorithm \ref{alg:cap_2} may increase. Given the assumption $L_1 < 2L$ and Remark 1 in \cite{zam}, we find that such an increase is not critical.

\begin{remark}
Now, for problem \eqref{main_composite_problem}, we will show that the Algorithm \ref{alg:cap_2} is universal, namely, we will prove that this method is applicable not only to relatively smooth but also to relatively Lipschitz continuous optimization problems. 

In \cite{adapt_alg} it is shown that if the function $f$ is relatively Lipschitz continuous, then for any fixed $L>0$ and $\delta>0$ the following inequality holds
\begin{equation*}
    f(x)\leq f(y)+\langle \nabla f(y), x - y\rangle + L V(x,y) + L V(y,x) + \delta,  \quad \forall x,y\in Q.
\end{equation*}
Therefore, for $L_{k+1}\geq L$, taking into account the triangular scaling property for Bregman divergence, we obtain
\begin{align}
    f(x_{k+1}) & \leq f(y_{k})+\langle \nabla f(y_{k}), x_{k+1}-y_{k}\rangle + L_{k+1} V(x_{k+1},y_{k}) +L_{k+1} V(y_{k},x_{k+1})+\delta \nonumber
    \\& = f(y_{k})+\left\langle \nabla f(y_{k}),(1-\theta_k)x_k + \theta_k z_{k+1} - (1-\theta_k)x_k-\theta_k z_{k}\right\rangle  \nonumber
    \\& \quad + L_{k+1} V \left((1-\theta_k)x_k + \theta_k z_{k+1}, (1-\theta_k) x_k + \theta_k z_{k} \right)  \nonumber
    \\& \quad + L_{k+1} V \left((1-\theta_k)x_k + \theta_k z_{k},(1-\theta_k)x_k+\theta_k z_{k+1} \right) + \delta \nonumber
    \\& \leq f(y_{k})+\theta_k\langle \nabla f(y_{k}), z_{k+1}-z_{k}\rangle + L_{k+1}\theta^{\gamma}_k V\left(z_{k+1},z_{k}\right)+L_{k+1}\theta^{\gamma}_k V\left(z_{k},z_{k+1}\right) \nonumber
    \\& \quad + \delta \nonumber
    \\& = f(y_{k}) +\theta_k \Bigg( \langle \nabla f(y_{k}), z_{k+1}-z_{k} \rangle + L_{k+1} \theta^{\gamma-1}_k V(z_{k+1}, z_k)  \nonumber
     \\& \qquad \qquad \qquad \quad + L_{k+1} \theta^{\gamma-1}_kV(z_k, z_{k+1})\Bigg) + \delta. \label{fdsd54d}
\end{align}

Let us consider as an $(\delta, L)$-oracle of the function $f$ the pair $(f_\delta (y), \nabla f_\delta (y)) = (f(y), \nabla f(y))$. Since 
\begin{equation*}
    z_{k+1}:=\arg\min\limits_{x\in Q} \{f(y_k) + \langle \nabla f(y_{k}), x - y_{k}\rangle + \theta^{\gamma -1}_k L_{k+1}V(x,z_k)\}, 
\end{equation*}
then from Lemma \ref{def4} it holds the following inequality 
\begin{equation*}
    \langle \nabla f(y_{k}), z_{k+1}-x\rangle \leq L_{k+1}\theta^{\gamma-1}_kV(x,z_k)-L_{k+1}\theta^{\gamma-1}_kV(z_{k+1},z_k)-L_{k+1}\theta^{\gamma-1}_kV(x,z_{k+1}).
\end{equation*}

Thus, for $x=z_k$ we have
\begin{equation*}
    \langle \nabla f(y_{k}), z_{k+1}-z_k\rangle \leq -L_{k+1}\theta^{\gamma-1}_kV(z_{k+1},z_k)-L_{k+1}\theta^{\gamma-1}_kV(z_k,z_{k+1}).
\end{equation*}
Therefore, from \eqref{fdsd54d}, we have
\begin{equation*}
    f(x_{k+1})\leq f(y_{k})+\delta.
\end{equation*}

Since for relatively Lipschitz continuous functions the following inequality also holds
\begin{equation*}
    \langle \nabla f(y), x-y\rangle + LV(x,y)+\delta\geq 0,  \quad \forall x,y\in Q.
\end{equation*}
Then
\begin{equation*}
    f(x_{k+1})\leq f(y_{k}) + \delta\leq f(y_{k}) +\langle \nabla f(y_{k}), x_{k+1}- y_{k} \rangle + L V(x_{k+1}, y_{k}) + 2 \delta.
\end{equation*}

Thus, the exit criterion from the iteration of the Algorithm \ref{alg:cap_2} will be guaranteed to be satisfied when $L_{k+1}\geq L$ and $\delta_{k+1}\geq 2 \delta$. 
\end{remark}

\section{Adaptive intermediate Bregman method for relatively smooth functions}\label{sec:AIBM}

In this section, we consider the following problem
\begin{equation}\label{main_smooth_problem}
    \min_{x \in Q} f(x), 
\end{equation}
where $f: Q \longrightarrow \mathbb{R}$ is a relatively $L$-smooth function (for some $L > 0$) (see \eqref{eq_6}). 

We assume that the function $f$ is equipped with a first-order $(\delta, L)$-oracle (see \cite{stonyakin2021inexact}), i.e.,  we can compute the pair $(f_{\delta}(y), \nabla_{\delta}f(y)) \in \mathbb{R}\times \mathbb{E}^*$, for any $y \in Q$,  such that
\begin{equation}\label{oracle_smooth}
    0 \leq f(x) - \underbrace{\left(f_{\delta} (y) + \langle \nabla_{\delta} f(y), x - y  \rangle\right)}_{ : = l_{\delta}(y,x)} \leq L V(y, x) + \delta, \quad \forall x \in Q. 
\end{equation}

For problem \eqref{main_smooth_problem}, based on Algorithm 1 from \cite{kamzolov2021universal}, we consider an adaptive algorithm called Adaptive Intermediate Bregman Method (AIBM), which is listed as Algorithm \ref{alg:AIBM} below. The adaptation in this algorithm is performed for the relative smoothness parameter $L$. The idea behind the procedure consists in adaptively constructed searching for a constant $L_k$, that satisfies \eqref{oracle_smooth} for a varying time sequence of inexactness $\{\delta_k\}_{k \geq 0}$. For the relatively $L$-smooth functions, the search for $L_k$ halts in finite time when $\gamma \geq 1$. 

Note that, for the considered class of functions in this section, it was proved that the largest value of $\gamma$, in \eqref{eq_15}, cannot exceed $2$ \cite{hanzely2021accelerated}.  Thus, we will assume that $\gamma \in (1, 2]$ in AIBM to be applicable to solve problem \eqref{main_smooth_problem}.

Hereinafter we assume that for the sequences $\{\alpha_k\}_{k \geq 0}, \{A_k\}_{k \geq 0}, \{B_k\}_{k \geq 0}$, generated by Algorithm \ref{alg:AIBM}, it holds the following inequalities
\begin{equation} \label{eq:order}
    0 < \alpha_k \leqslant B_k \leqslant A_k.
\end{equation}

Let us set 
\begin{equation}\label{Psi_smooth_case}
    \Psi_k(x) : = d(x) + \sum_{i = 0}^{k} \alpha_i l_{\delta_i}(x_i, x), \quad \forall k \geq 0.  
\end{equation}

For this function, we mention the following lemma which provides a lower bound for it. 

\begin{lemma}[Lemma 5.5.1, \cite{ben2001lectures}]
If $x = \arg \min_{x \in Q} \Psi_k(x)\; \forall k \geq 0$, then the following inequality holds
\begin{equation} \label{ineq:psi_smooth}
    \Psi_k(y) \geq \Psi_k(x) + V(y, x), \quad \forall y \in Q.
\end{equation}
\end{lemma}

\begin{algorithm}[htp]
\caption{Adaptive Intermediate Bregman Method (AIBM). }\label{alg:AIBM}
\begin{algorithmic}[1]
    \STATE \textbf{Input:} positive sequence $\{\delta_k\}_{k \geq 0}$, prox-function $d$, iterations number $N$, $L_0 > 0$, intermediate parameter $p \in [1, 2]$, $\gamma \in (1, 2]$. 
    \STATE $x_0 = z_0 = \arg \min_{x \in Q} d(x),$ and set $ L_0 = L_0 / 2.$  
    \WHILE{\label{1line:sc1} {$f_{\delta_0} (y_0)  > l_{\delta_0} (x_0, y_0) + L_0 V(y_0, x_0) + \delta_0$}}
    \STATE $L_0 = 2  L_0$, $\alpha_0 = 1/L_0,$ \label{1line:alpha1} 
    \STATE $y_0 = \arg \min_{x \in Q}\Psi_0(x).$ \label{1line:yo} 
    \ENDWHILE
    \STATE Set $B_0 = A_0 = \alpha_0 = 1/ L_0.$
    \FOR{$k = 1,2, \ldots, N$}
        \STATE $L_k = L_{k-1} / 2.$
        \WHILE{\label{1line:sc2} {$f_{\delta_k} (w_k) > l_{\delta_k} (x_k, w_k) + L_k V(w_k, x_k) + \delta_k$}}
            \STATE $L_k = 2  L_k$, $\alpha_k = \frac{1}{L_k} \left(1 + \frac{k}{2p}\right)^{(p-1)(\gamma-1)},$ \label{1line:alphas} 
            \STATE $B_k = \left(L_k \alpha_k^\gamma\right)^{1/(\gamma - 1)},$  \label{1code:B}
            \STATE $x_k = \frac{\alpha_k}{B_k} z_{k-1} + (1 - \frac{\alpha_k}{B_k}) y_{k-1},$ \label{1line:x} 
            \STATE $z_k = \arg\min_{x \in Q} \Psi_k(x),$ \label{1line:z} 
            \STATE $w_k = \frac{\alpha_k}{B_k} z_k + (1 - \frac{\alpha_k}{B_k}) y_{k-1},$ \label{1line:w}
            \ENDWHILE
        \STATE $A_k = A_{k-1} + \alpha_k,$ \label{1line:A} 
        \STATE $y_k = \frac{B_k}{A_k} w_k + (1 - \frac{B_k}{A_k}) y_{k-1}.$ \label{1line:y} 
    \ENDFOR
\end{algorithmic}
\end{algorithm}

For proving the convergence rate of the proposed Algorithm \ref{alg:AIBM} (AIBM), let us first prove the following useful lemma. 

\begin{lemma}\label{lemma_upper_bound_Ak_fk}
Let $f$ be a convex and relatively $L$-smooth function, equipped with a first-order $(\delta, L)$-oracle. Then by AIBM, it holds the following inequality 
\begin{equation}\label{upper_bound_Ak_fk}
    A_k f(y_k) - E_k \leq \Psi^*_k, \quad \forall k \geq 0,
\end{equation}
where $\Psi^*_k = \min_{x \in Q} \Psi_k (x),$ and $ E_k = \sum_{i = 0}^{k} B_i \delta_i$. 
\end{lemma}

\begin{proof}
We prove this lemma by induction. For $k = 0$, we have
\begin{align*}
    \Psi_0^* & \;\; =\;\; \min_{x \in Q} \left\{ \Psi_0 (x) \right\}
    = \min_{x \in Q} \left\{ d(x) + \alpha_0 f_{\delta_0}(x_0) + \alpha_0 \left\langle g_{\delta_0}(x_0), x - x_0 \right\rangle\right\}
    \\& \stackrel{\text{Line} \, \ref{1line:yo}}{\geq} d(y_0) + \alpha_0 l_{\delta_0}(x_0, y_0) =  \alpha_0 \left(L_0 V(y_0, x_0) +  l_{\delta_0}(x_0, y_0)\right)
    \\& \stackrel{\text{Line} \, \ref{1line:sc1}}{\geq}  \alpha_0 f_{\delta_0}(y_0) - \delta_0 - \alpha_0 \delta_0 \geq A_0 f_{\delta_0}(y_0) - B_0 \delta_0  = A_0 f(y_0) - E_0.
\end{align*}

Assume that \eqref{upper_bound_Ak_fk} is valid for certain $k-1 \geq 0$, i.e.,
\begin{equation}\label{induction_k_1}
    A_{k-1} f(y_{k-1}) - E_{k-1} \leq \Psi_{k-1}^* = \min_{x \in Q} \Psi_{k-1}(x),
\end{equation}
and let us prove that it holds for $k$. Indeed, 
\begin{align*}
    \Psi_k^* & \;\; =\; \; \min_{x \in Q} \Psi_k(x_k) \stackrel{\text{Line} \, \ref{1line:z}}{\geq}\Psi_k(z_k) \stackrel{\eqref{Psi_smooth_case}}{=} \Psi_{k-1}(z_k) + \alpha_k l_{\delta_k} (x_k, z_k)
    \\&  \stackrel{\eqref{ineq:psi_smooth}}{\geq} \Psi_{k-1}(z_{k-1}) + V(z_k, z_{k-1}) + \alpha_k l_{\delta_k}(x_k, z_k)
     \geq \Psi_{k-1}^* + V(z_k, z_{k-1}) + \alpha_k l_{\delta_k}(x_k, z_k)
    \\& \stackrel{\eqref{induction_k_1}} {\geq} A_{k-1} f(y_{k-1}) - E_{k-1} + V(z_k, z_{k-1}) + \alpha_k l_{\delta_k}(x_k, z_k) 
    \\& \stackrel{ \text{Line} \, \ref{1line:A} }{\geq} (A_k - \alpha_k) f(y_{k-1}) - E_{k-1} + V(z_k, z_{k-1})  + \alpha_k l_{\delta_k}(x_k, z_k) 
    \\& = (A_k - B_k) f(y_{k-1}) - E_{k-1} + V(z_k, z_{k-1}) + (B_k - \alpha_t)f(y_{k-1}) + \alpha_k l_{\delta_k}(x_k, z_k) 
    \\& \stackrel{\eqref{oracle_smooth}}{\geq} (A_k - B_k) f(y_{k-1}) - E_{k-1} + V(z_k, z_{k-1}) + (B_k - \alpha_k)l_{\delta_k}(x_k, y_{k-1})  
    \\& \quad + \alpha_k l_{\delta_k}(x_k, z_k) 
    \\& = (A_k - B_k) f(y_{k-1}) - E_{k-1}  + V(z_k, z_{k-1})  + B_k f(x_k)  
    \\& \quad + \left\langle g_{\delta_k}(x_k), (B_k - \alpha_k)(y_{k-1} - x_k)+ \alpha_k (z_k - x_k) \right\rangle.
\end{align*} 

From Line \ref{1line:x} in AIBM we get
\begin{equation*}
    (B_k- \alpha_k)(y_{k-1} - x_k)+ \alpha_k (z_k - x_k)  = \alpha_k (z_k - z_{k-1}).
\end{equation*}

Thus, we conclude
\begin{align*}
    \Psi_k^* &\geq (A_k - B_k) f(y_{k-1}) - E_{k-1} + V(z_k, z_{k-1}) +  B_k f(x_k) + \alpha_k \left\langle g_{\delta_k}(x_k), z_k - z_{k-1} \right\rangle
    \\& = (A_k - B_k ) f(y_{k-1}) - E_{k-1} 
    \\& \quad + B_k \left( f_{\delta_k}(x_k) + \frac{1}{B_k} V(z_k, z_{k-1})  + \left\langle g_{\delta_k}(x_k), z_k - z_{k-1} \right\rangle \right).
\end{align*}

Because of $\frac{1}{B_k} = \frac{L_k \alpha_k^{\gamma}}{B_k^{\gamma}}$ (from Line \ref{1code:B} in AIBM), and $\frac{\alpha_t}{B_k} (z_k - z_{k-1}) = w_k - x_k$ (from Lines \ref{1line:x} and \ref{1line:w} in AIBM), we get
\begin{align*}
    \Psi_k^* & \geq  (A_k - B_k ) f(y_{k-1}) - E_{k-1}
    \\& \quad + B_k \left( f_{\delta_k}(x_k) + L_k \left(\frac{\alpha_k}{B_k}\right)^{\gamma} V(z_k, z_{k-1})  + \left\langle g_{\delta_k}(x_k), w_k - x_k \right\rangle \right).
\end{align*}

From Lines \ref{1line:x} and \ref{1line:w} in AIBM, due to $\frac{\alpha_k}{B_k} \leq 1, \forall k \geq 0$ (see \eqref{eq:order})  
and \eqref{eq_15}, we have 
\begin{align*}
    V(w_k, x_k) & = V\left(\frac{\alpha_k}{B_k} z_k + \left(1- \frac{\alpha_k}{B_k}\right)y_{k-1}, \frac{\alpha_k}{B_k} z_{k-1} + \left(1-\frac{\alpha_k}{B_k}\right)y_{k-1}\right)
    \\& \leq \left(\frac{\alpha_k}{B_k}\right)^{\gamma} V(z_k, z_{k-1}).
\end{align*}

Thus, we get
\begin{align*}
    \Psi_k^* & \;\;\; \geq \;\;  (A_k - B_k) f(y_{k-1}) -  E_{k-1} + B_k \left(l_{\delta_k}(x_k, w_k) + L_k V(w_k, x_k)\right)
    \\& \stackrel{\text{Line} \, \ref{1line:sc2}}{\geq}  (A_k - B_k) f(y_{k-1}) -  E_{k-1}  + B_k\left(f_{\delta_k}(w_k) - \delta_k \right). 
\end{align*}

From line \ref{1line:y} in AIBM, we find $A_k y_k = B_k w_k + (A_k - B_k) y_{k-1}$. Thus, since $f$ is convex, we get the following
\[
    \Psi_k^* \geq A_k f(y_k) - E_{k-1} - B_k \delta_k  = A_k f(y_k) - E_{k},
\]
which completes the proof. 
\end{proof}

From Lemma \ref{lemma_upper_bound_Ak_fk}, we can conclude the following result.
\begin{corollary}\label{corr_rate_adaptive}
Let $f$ be a convex and relatively $L$-smooth function, equipped with a first-order $(\delta, L)$-oracle. Then by AIBM, we have 
\begin{equation}\label{adaptive_estimate}
    f(y_k) - f(x_*) \leq \frac{d(x_*)}{A_k} + \frac{1}{A_k} \sum_{i = 0}^{k}B_i \delta_i, \quad \forall k \geq 0, 
\end{equation}
where $x_* = \arg\min_{x \in Q} f(x)$. 
\end{corollary}

\begin{proof}
From Lemma \ref{lemma_upper_bound_Ak_fk}, we have
\begin{align*}
    A_k f(y_k) - E_k & \; \leq \Psi_k^* =   \min _{x \in Q}\left\{d(x) + \sum_{i = 0}^{k} \alpha_i l_{\delta_i}(x_i, x) \right\} 
    \leq  d(x_*) + \sum_{i = 0}^{k} \alpha_i l_{\delta_i}(x_i, x_*)
    \\&  \stackrel{\eqref{oracle_smooth}}{\leq} d(x_*) +\sum_{i = 0}^{k} \alpha_i f(x_*) = d(x_*) + A_k f(x_*).
\end{align*}

This means, 
\[
    A_k f(y_k) - E_k \leq d(x_*) + A_k f(x_*),
\]
i.e.,
\[
    A_k \left(f(y_k) - f(x_*)\right) \leq d(x_*) + E_k = d(x_*) + \sum_{i= 0}^{k} B_i \delta_i.
\]

Therefore, 
\[
    f(y_k) - f(x^*) \leq \frac{h(x_*)}{A_k} + \frac{1}{A_k} \sum_{i= 0}^{k} B_i \delta_i, 
\]
which completes the proof. 
\end{proof}

We will prove the following theorem to evaluate the convergence rate of  Algorithm \ref{alg:AIBM} (AIBM).

\begin{theorem}\label{the:rate_adaptive_alg}
Bregman divergence $V(\cdot, \cdot)$ satisfies the triangular scaling property with scaling factor $\gamma \in (1, 2]$. Let $f$ be a convex and relatively $L$-smooth function, equipped with a first-order $(\delta, L)$-oracle. Let  $p \in [1,2]$ be the intermediate parameter, $\gamma \in (1, 2]$, and let us assume that $R_0 > 0$  is an upper bound for the distance to a solution $x_*$ from the starting point $x_0$, such that $V(x_0, x_*) = d(x_*) \leq R_0$. Then, for AIBM, the following inequality holds for any $k \geq 0$
\begin{equation}\label{rate:AIBM}
    f(y_k) - f(x_*) \leq \frac{ 16  R_0 }{(k+2)^{(p-1)(\gamma - 1)+1}} \left(\max_{0 \leq i \leq k} L_i\right)  + (k+2p)^{p-1}\left(\max_{0 \leq i \leq k} \delta_i\right).
\end{equation}
\end{theorem}

\begin{proof}

First, for $\{\alpha_k\}_{k \geq 0}$ and  $ \{B_k\}_{k \geq 0}$ we have (see \eqref{eq:order}) 
\[
    0 < \alpha_{k+1} \leq B_{k+1} \leq A_{k+1}.
\]

Let us find the lower bound for $A_k$ and the upper bound for $\sum_{i = 0}^{k}B_{i} \delta_i$. 
\begin{itemize}
\item 
For $A_k$, we have
\begin{equation*}
    \alpha_k = \frac{1}{L_k} \left(\frac{k + 2p}{2p} \right)^{(p-1)(\gamma - 1)} \geq \frac{1}{\max_{0 \leq i \leq k} L_i} \left(\frac{k + 2p}{2p} \right)^{(p-1)(\gamma - 1)}.
\end{equation*}
Thus, 
\[
    A_k = \sum_{i = 0}^{k} \alpha_i \geq \frac{1}{\max_{0 \leq i \leq k} L_i} \sum_{i = 0}^{k} \left(\frac{i + 2p}{2p} \right)^{(p-1)(\gamma - 1)}.
\]
But
\begin{align*}
     \sum_{i = 0}^{k} \left(\frac{i+2p}{2p}\right)^{(p-1)(\gamma - 1)} & \geq \int_{0}^{k} \left(\frac{x+2p}{2p}\right)^{(p-1)(\gamma - 1)} dx + \alpha_0 
    \\& \geq \frac{2}{(p-1)(\gamma - 1) + 1} \left(\frac{k+2p}{2p}\right)^{(p-1)(\gamma - 1) + 1}
    \\& \geq \left(\frac{k+2p}{2p}\right)^{(p-1)(\gamma - 1) + 1}.
\end{align*}
Therefore, since $p \in [1,2]$, we get
\begin{align}\label{lower_bound_Ak}
    A_k & \geq \frac{1}{ \max_{0 \leq i \leq k} L_i}  \left(\frac{k+2p}{2p}\right)^{(p-1)(\gamma - 1) + 1}  \nonumber 
    \\& \geq \frac{1}{\max_{0 \leq i \leq k} L_i}  \left(\frac{k+2}{4}\right)^{(p-1)(\gamma - 1) + 1}.
\end{align}

\item 
For $\sum_{i=0}^{k}B_i \delta_i$, from line \ref{1line:alphas} in AIBM, we have
\begin{align*}
    B_i & = \left(L_i \alpha_i^{\gamma}\right)^{1/(\gamma - 1)} = L_i^{1/(\gamma - 1)} \alpha_i^{1/(\gamma - 1)} \alpha_i
    \\& = L_i^{1/(\gamma - 1)} \left(\frac{1}{L_i} \left(1+ \frac{i}{2p}\right)^{(p-1)(\gamma - 1)} \right)^{1/(\gamma - 1)} \alpha_i = \left(\frac{i + 2p}{2p}\right)^{p-1} \alpha_i. 
\end{align*}

Thus, 
\begin{align}\label{upper_Bidi}
    \sum_{i=0}^{k}B_i \delta_i & = \sum_{i=0}^{k} \left(\frac{i + 2p}{2p}\right)^{p-1} \alpha_i \delta_i \leq \left(\frac{k + 2p}{2p}\right)^{p-1} \left(\max_{0 \leq i \leq k}  
    \delta_i\right)   \sum_{i=0}^{k} \alpha_i        \nonumber 
    \\&  =  \left(\frac{k + 2p} {2p}\right)^{p-1}  \left(\max_{0 \leq i \leq k} \delta_i\right) A_k. 
\end{align}
\end{itemize}

From \eqref{lower_bound_Ak}, \eqref{upper_Bidi} and Corollary \ref{corr_rate_adaptive}, we have 
\begin{align*}
    f(y_k) - f(x_*) & \leq \frac{4^{(p - 1)(\gamma - 1) + 1}d(x_*)}{(k+2)^{(p - 1)(\gamma - 1) + 1}}\left(\max_{0 \leq i \leq k} L_i\right)  + \frac{1}{A_k} \left(\frac{k + 2p} {2p}\right)^{p-1}  \left(\max_{0 \leq i \leq k} \delta_i\right) A_k
    \\& \leq \frac{16 R_0}{(k+2)^{(p - 1)(\gamma - 1) + 1}}\left(\max_{0 \leq i \leq k} L_i\right)  + (k+2p)^{p-1}\left(\max_{0 \leq i \leq k} \delta_i\right),
\end{align*}
which is the desired convergence rate of AIBM.   
\end{proof}

\begin{remark}\label{rem:exact_setting}
Note, that when we consider the Euclidean setting of the problem, i.e., $d(x) = \frac{1}{2}\|x\|_2^2, V(x,y) = \frac{1}{2}\|x-y\|_2^2$ and the uniform triangle scaling exponent $\gamma =2$ \cite{hanzely2021accelerated}, for problem \eqref{main_smooth_problem} we get the following convergence rate of AIBM
\begin{equation*}
    f(y_k) - f(x_*)  \leq \frac{16 R_0}{(k+2)^{p}}\left(\max_{0 \leq i \leq k} L_i\right) + (k+2p)^{p-1}\left(\max_{0 \leq i \leq k} \delta_i\right), \quad \forall k \geq 0. 
\end{equation*}
\end{remark}

\begin{remark}\label{rem:alg_with_var_p}
For the problem \eqref{main_smooth_problem}, depending on \eqref{rate:AIBM}, we can reconstruct AIBM and propose an adaptive algorithm with the adaptation to two parameters $L$ and $p$ (see Algorithm \ref{Adap_Lp:AIBM}). 
In this algorithm, with a fixed value of $\gamma \in (1,2]$ (depending on the prox-function $d$ and the setting of the problem), we start from $p = 2$, which is the case for which the estimate \eqref{rate:AIBM} is the best at the first iterations. Then at the iterations, when the estimate \eqref{rate:AIBM} deteriorates we decrease $p$ until $p=1$. As a result, Algorithm \ref{Adap_Lp:AIBM} ensures a solution to the problem  \eqref{main_smooth_problem} with an estimate better than the estimate with fixed $ p \in [1, 2]$ as in AIBM (Algorithm \ref{alg:AIBM}). 

\begin{algorithm}[H]
\caption{Adaptive Intermediate Bregman Method in both parameters $L$ and $p$ (AIBM with variable $p$).}\label{Adap_Lp:AIBM}
\begin{algorithmic}[1]
   \STATE \textbf{Input:} initial point $x_0$ such that $V(x_0, x_*) \leq R_0$, $L_0 > 0$, positive sequence $\{\delta_k\}_{k \geq 0}, \gamma \in (1,2]$, small enough parameter $\eta \ll 1$.
  \STATE Set $p_0:= 2$, and  $S_0 = \frac{16 R_0 L_0}{2^{\gamma}} + 4 \delta_0$.
   \FOR{$k = 0, 1,   \ldots$}
   \STATE Run AIBM (Algorithm \ref{alg:AIBM}) with parameters $x_k, L_k$, and  $p_k$, i.e., run AIBM $(x_k, L_k, p_k)$.
   \STATE
   Calculate
   \[
        S_{k+1} : = \frac{ 16 R_0 }{(k+3)^{(p_k - 1)(\gamma - 1)} +1} \left(\max_{0 \leq i \leq k+1} L_i \right) +  (k+1 + 2 p_k)^{p_k-1}\left(\max\limits_{0 \leq i \leq k+1} \delta_i\right).
   \]
   \STATE  \textbf{If} $S_{k+1} \leq S_{k}$, \textbf{then} go to the next iteration $k \longrightarrow k+1$. \textbf{Else} run AIBM $(x_k, L_k, p_k - \eta)$. 
   \ENDFOR
\end{algorithmic}
\end{algorithm}
\end{remark}

\begin{remark}
Considering an exact oracle, i.e., $\delta_k = 0 \, (\forall k \geq 0)$, thus $f_{\delta}(x) = f(x), \nabla_{\delta} f(x) = \nabla f(x)$, for any $x \in Q$, we get the following convergence rate of AIBM
\[
    f(y_k) - f(x_*) \leq \frac{16 R_0}{(k+2)^{(p - 1)(\gamma - 1) + 1}}\left(\max_{0 \leq i \leq k} L_i\right) , \quad \forall k \geq 0.
\]
Thus, the intermediate method exhibits intermediate behaviors, and the performance of the algorithm increases when $p$ increases between $1$ and $2$ for fixed values of $\gamma \in (1, 2]$.
\end{remark}

\section{Numerical Experiments}\label{sect:numerical_exper}

In this section, to demonstrate the performance of the proposed Algorithms \ref{alg:cap} (AdapFGM), \ref{alg:cap_1} (AccBPGM-1), \ref{alg:cap_2} (AccBPGM-2), and \ref{alg:AIBM} (AIBM) we conduct some numerical experiments for the Poisson linear inverse problem. 

Let us consider the Poisson inverse problem \eqref{problem_invPoisson} (see Example \ref{ex:InvPois}) with $\psi(x) = 0$. At the first, we compare the proposed algorithms AdapFGM, AccBPGM-1, and AccBPGM-2 with the non-accelerated adaptive Bregman proximal gradient method (BPG-Adapt) \cite{BPG}. We run these algorithms by taking $m = 150, n = 100$, and $x_0$ is an initial random generated point near the center of the feasible region $Q$. The matrix $A \in \mathbb{R}^{m \times n}$ and the vector $b \in \mathbb{R}_{++}^{m}$ are randomly generated by the uniform distribution on the interval $[0, 1]$.  The results of the work of these compared algorithms are represented in Fig. \ref{fig_Poisson1} below. These results demonstrate the difference $F(x_k) - F_*$, where $F_* = F (x_*)$ is the optimal value of the objective function, with different values of the triangular scaling factor $\gamma$.   

\begin{figure}[htp]
\centering
{\resizebox*{\columnwidth}{!}{\includegraphics{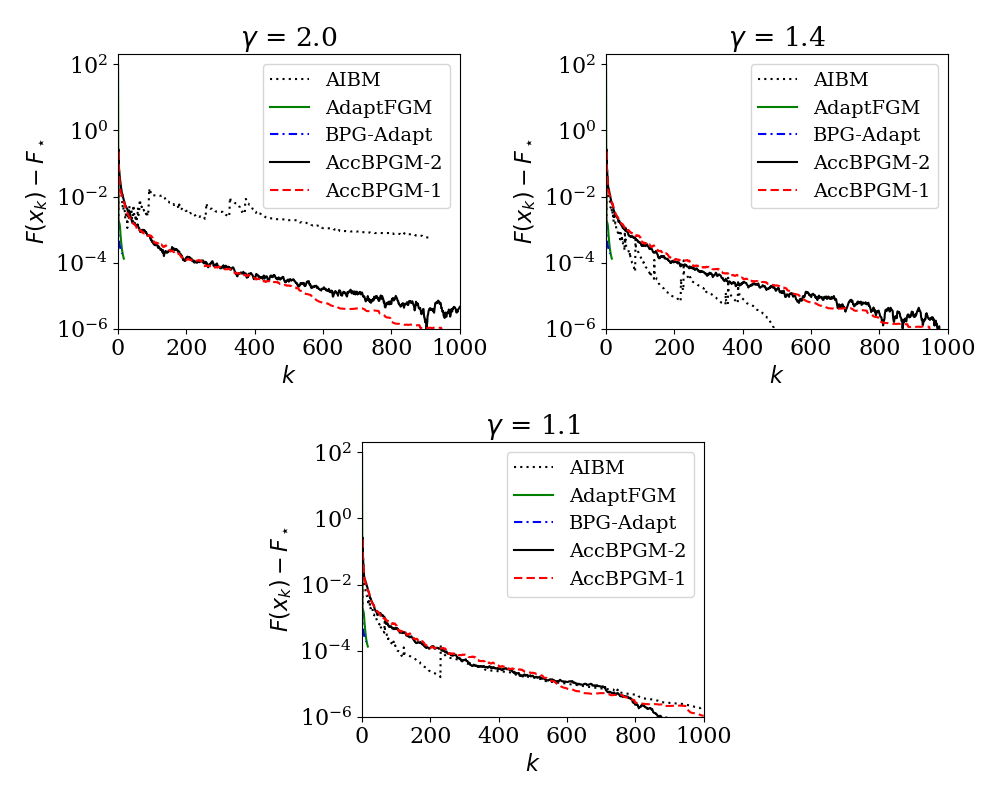}}}
  \caption{The results of Algorithms \ref{alg:cap} (AdapFGM), \ref{alg:cap_1} (AccBPGM-1), \ref{alg:cap_2} (AccBPGM-2), \ref{alg:AIBM} AIBM, and BPG-Adapt \cite{BPG}, for Poisson inverse problem with $m = 150, n = 100$ and and different values of $\gamma$. }
\label{fig_Poisson1}
\end{figure}

In Fig. \ref{fig_Poisson1}, when $\gamma = 2$, we can see that AccBPGM-1 showed the best results. Also, all the proposed algorithms give better results than the usual BPG except the first plot where AIBM shows poor convergence. Whereas with a decreasing $\gamma$, when $\gamma = 1.4$, it can be seen that the quality of Algorithms AccBPGM-1 and AccBPGM-2, in the first iterations, begin to deteriorate. But when $\gamma = 1.1$, we find that as $\gamma$ decreases, the quality of AccBPGM-1 and AccBPGM-2 decreases and the AIBM gives the best result. From the graphs above, we observe that the AdaptFGM algorithm demonstrates very high convergence speed in the initial iterations but quickly ceases to progress, stopping too early.

\bigskip 

%
%

Now let us introduce interference represented by a random variable $\delta_k > 0$ and check how it will affect the convergence of the AIBM for the same Poisson inverse problem. The random variable will follow a uniform distribution with different expectations. The results are shown in Fig. \ref{poisson_prob_noise}. We can clearly see that the former advantage of the algorithm is preserved, we can also see that when $\gamma$ decreases, the influence of interference decreases.

\begin{figure}[h]
    \includegraphics[width=\columnwidth]{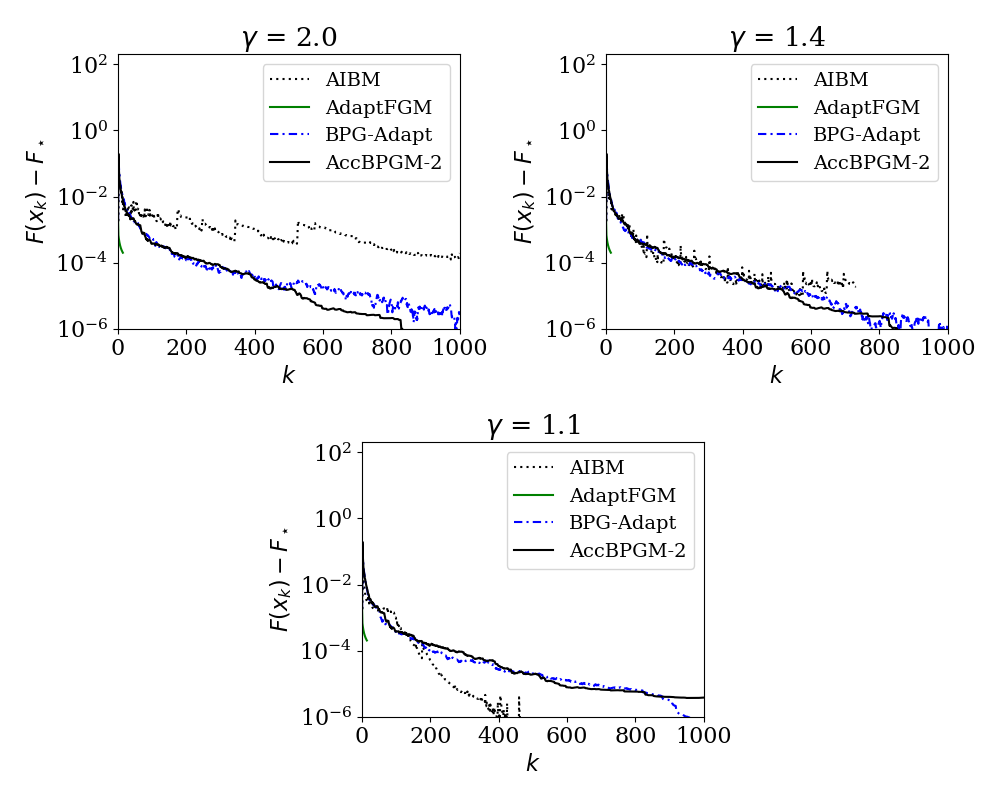}
    \centering
    \caption{The results of the compared algorithms, with noise, for the Poisson inverse problem. All other problem parameters are still the same as in the previous experiment.}
    \label{poisson_prob_noise}
\end{figure}

Further, we again consider a random value $\delta > 0$ and check how it will affect the convergence of the AIBM for different values of $\delta$. The random variable will follow a uniform distribution with different expectations. Fig. \ref{noise_experiments} shows the result with a uniformly distributed random variable $\delta$. As we can see, although the level of expectation $\delta$ affects the amount of deviations of the algorithm values, the dynamics of the algorithm convergence remain the same. 

\begin{figure}[h]
    \includegraphics[width=\columnwidth]{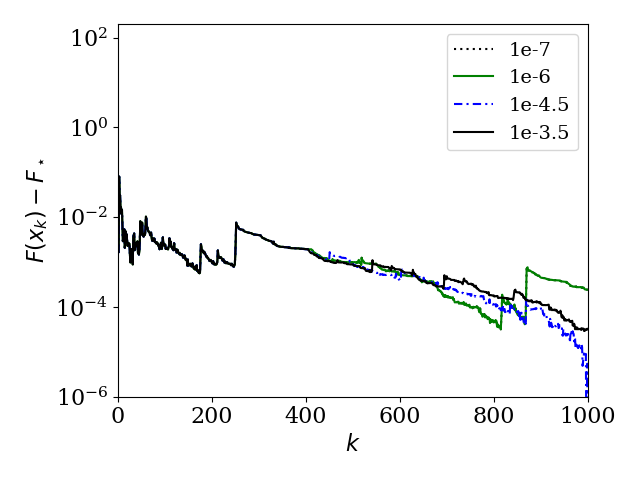}
    \centering
    \caption{The results of the AIBM, with different $\delta$ expectancy, uniform distribution, for the Poisson inverse problem. All other problem parameters are still the same as in the previous experiment.}
    \label{noise_experiments}
\end{figure}

\section{Conclusion}

In this paper, we proposed some accelerated methods for solving the problem of minimizing the convex differentiable and relatively smooth function.  The first proposed method is an adaptive fast gradient method based on a similar triangles method with an inexact oracle, which uses a special triangular scaling property, for the used Bregman divergence.   The other two proposed methods are an adaptive accelerated Bregman proximal gradient method using an inexact oracle and its adaptive version.  The adaptivity in all proposed methods is for the relatively smoothness parameter.  We concluded the convergence rate of all proposed methods and proved that they are universal in the sense that they are applicable not only to relatively smooth but also to relatively Lipschitz continuous optimization problems.  We also proposed an automatically adaptive intermediate Bregman method which interpolates between slower but more robust algorithms non-accelerated and faster, but less robust accelerated algorithms. In the end, we introduced some results of the numerical experiments that were conducted, demonstrating the advantages of using the proposed algorithms for the Poisson inverse problem. Despite the accumulation of the inexactness, as shown in the estimation of the convergence rate of the proposed algorithms, the experimental results for the proposed methods were better than those of the accelerated and non-accelerated variants of the proximal gradient Bregman method. The best performance compared to other approaches was shown by the adaptive version of the accelerated proximal gradient Bregman method with an inexact oracle. With increasing the value of the triangular scaling parameter, the accuracy and convergence of accelerated proximal methods improved. 

\section*{Disclosure statement}
The authors declare no conflicts of interest.

\section*{Funding}
The research in Sections 3, 4, and 6  was supported by the Ministry of Science and Higher Education of the Russian Federation as part of the State Assignment No. 075-03-2024-074 for the project "Investigation of asymptotic characteristics of oscillations in differential equations and systems, as well as optimization methods".




\end{document}